%% file: Vardi.tex
\documentclass{au}
\usepackage{amssymb,latexsym}

    \usepackage[english]{babel}
    \usepackage{amsmath,amsthm}
    \usepackage{graphicx}
    \usepackage{wrapfig}
     \usepackage[all]{xy}
\theoremstyle{plain}
\newtheorem{theorem}{Theorem}
\newtheorem{corollary}{Corollary}

\newtheorem{lemma}{Lemma}

\theoremstyle{definition}

\theoremstyle{remark}

\numberwithin{equation}{section}

\newenvironment{Proof}[1][]
{ 
  \trivlist
  \item[\hskip\labelsep\itshape{\it Proof#1.}]\ignorespaces
}{ 
  \endtrivlist \par
}

\newcommand{\Pol}{
\mathrm{Pol}
}

\newcommand{\Inv}{
\mathrm{Inv}
}

\newcommand{\AND}{
\mathrm{And}
}

\newcommand{\Strike}{
\mathrm{Strike}
}

\begin{document}

\title[The existence of a near-unanimity function is decidable]
{The existence of a near-unanimity function is decidable}
\author{Dmitriy~N. Zhuk}
\address{Department of Mathematics and Mechanics\\
Moscow State University\\
Moscow, Russia 119192}
\email{zhuk@intsys.msu.ru}
\date{January 20, 2011}

\input{abstract}
\maketitle

\input{preamble}
\input{plan}

\input{mainResults}
\input{MainStatements}
\input{sush}
\input{GlavnoeUtverjdenie}

\input{transformations}
\input{proofs_main}
\input{biblio}

\end{document}

%% file: abstract.tex
\begin{abstract}

We prove that the following problem is decidable: given a finite set of relations, decide
whether this set admits a near-unanimity function.

\end{abstract}

%% file: preamble.tex
\section{Preamble}

We say that a function $f$ is a \emph{near-unanimity function}
iff the following condition holds:
$$f(x,y,y,\ldots,y) =
f(y,x,y,\ldots,y) = \ldots = f(y,y,\ldots,y, x) = y.
$$
Near-unanimity functions arise naturally in Clone Theory. 
For example, if a clone contains a near-unanimity function of arity $n$ then
this clone can be described by relations of arity $(n-1)$\cite{pixley,jeavons}.
Hence the number of clones containing a near-unanimity function of arity $n$ is finite.
Also, it can be proved that such clones have a finite basis.
So, clones containing a near-unanimity function can be completely described.
This approach for investigation of the lattice of clones in multi-valued logic is proposed in \cite{lau}.

Near-unanimity functions are also related to the constraint satisfaction problem.
The standard way to parameterize interesting subclasses of the constraint satisfaction
problem is via finite relational structures \cite{FederVardi}. The main
problem is to classify those subclasses that are tractable (solvable in polynomial
time) and those that are NP-complete. A related meta-problem is to decide for
a given finite relational structure whether or not it gives rise to a tractable
subclass of the constraint satisfaction problem.

It is proved in \cite{jeavons} that if a set $G$ of relations on a finite set
admits a near-unanimity function, then the corresponding constraint
satisfaction problem $CSP(G)$ is solvable in polynomial time. Therefore, it is natural
to consider \textbf{NUF-Problem}:
given a finite set of relations $G,$ decide
whether $G$ admits a near-unanimity function.

This problem was presented in \cite{Larose,Maroti2009} as an open problem.
Some partial results were obtained in solving this problem.
In \cite{Kun} Kun and Szab\'o described a polynomial-time algorithm that decides whether a finite poset
admits a near-unanimity function.
A polynomial-time algorithm that
recognizes reflexive, symmetric graphs admitting a near-unanimity function
is presented in~\cite{LaroseLoten}.

Mikl\'os Mar\'oti considered similar problem for functions.
Firstly, he proved that it is undecidable for finite sets $M$ of functions on a set $A$ and
two fixed elements $a, b \in A$
whether $[M]$ contains
a function that behaves as a near-unanimity function on $A\setminus \{a, b\} $\cite{Maroti2007}.
Then he showed that the following problem is decidable:
given a finite set of functions $M,$ decide
whether $[M]$ contains a near-unanimity function \cite{Maroti2009}.

In this paper we prove that \textbf{NUF-Problem} is decidable.
It turns out that to check that a finite set $G$ of relations admits a near-unanimity function
it is enough to check that $G$ admits a near-unanimity function of a calculable arity~$n(G).$

Also in this paper we present a natural reformulation of \textbf{NUF-Problem}.
A relation is called essential if it cannot be presented as a conjunction
of relations with smaller arities.
The set of all essential relations is denoted by $\widetilde R_{k}.$
It is proved that every closed set $G$ of relations
can be uniquely defined by the set of all essential relations from $G.$
Also we show that a set $G$ admits a near-unanimity function if and only if
the set $[G]\cap \widetilde R_{k}$ is finite.
So, \textbf{NUF-Problem} can be reformulated in the following natural way:
given a finite set $G$ of relations, decide whether
the set $[G] \cap \widetilde R_{k}$ is finite. 

%% file: plan.tex
\section{Structure of the paper}

This paper is organized as follows.

In Section 3 we give definitions of a predicate, a near-unanimity function,
\textbf{NUF-Problem}, and formulate the main result of this paper:
\textbf{NUF-Problem} is decidable.
The set of all predicates we denote by $R_{k}.$
It turns out that to check that a finite set $G$ of predicates admits a near-unanimity function
it is enough to check that $G$ admits a near-unanimity function of a calculable arity $n(G).$
We give a formula for $n(G)$ in Section 3.
Obviously, we can easily check that a finite set of predicates
admits a near-unanimity function of a fixed arity.
Hence, the problem "Does a relational structure admits NUF" is decidable.

In Section 4 we
give definitions of the Galois connection, an essential predicate, and reformulate \textbf{NUF-Problem}.
Essential predicates are all predicates that cannot be presented as a conjunction
of predicates with smaller arities.
The set of all such predicates we denote by $\widetilde R_{k}.$
We prove in Section 5 that $[[G]\cap \widetilde R_{k}]=[G]$
for every $G\subseteq R_{k}.$
This means that every closed set $G$ of predicates
can be described by the set $G\cap \widetilde R_{k}$ of all essential predicates from $G.$
So, the set of all essential predicates is strong enough.
Here we state that a set $G\subseteq R_{k}$ admits a near-unanimity function if and only if
the set $[G]\cap \widetilde R_{k}$ is finite.
So, \textbf{NUF-Problem} can be reformulated in the following natural way:
given a finite set $G\subseteq R_{k}$, decide whether
the set $[G] \cap \widetilde R_{k}$ is finite.
Also in this section we formulate several theorems and
derive the main result from these theorems.

Section 5 is devoted to essential predicates.
Here we prove a theorem from Section 4
and present other important properties of essential predicates, which are 
used in the following sections.

In Section 6 we prove an important equation that is used in the next sections.
The idea of this equation is following:
if $[\{\rho\}]\cap \widetilde R_{k}$ is finite
(a predicate $\rho$ admits a near-unanimity function)
and $\rho'$ is obtained from $\rho$
by identification of variables, then $\rho'$ can be obtained from $\rho$ without identification of variables.

Section 7 is the most complicated part of this paper.
Here we define notions that allow us to transform formulas.
Roughly speaking we consider formulas as graphs.
We give definitions of a path in a formula, a connected formula, a tree-formula and so on.
The main part of this section is devoted to an important transformation of formulas.
This transformation is based on the equation from Section 6.
It allows us to remove identification of variables form formulas.
Also we define different characteristics of formulas and show
how the transformation changes these characteristics.
At the end of this section, we prove that if $[G]\cap \widetilde R_{k}$ is finite
and has a maximal predicate, then this predicate can be realized by a tree-formula.

In the last section we prove the main theorems.
Here we consider only formulas that realize essential predicates.
The first theorem claims that if we have a tree-formula, then we can obtain a
chain-formula that realizes an essential predicate.
Moreover, the size of the chain-formula depends monotonically on the size of the tree-formula.
The next theorem of this section states that
if we have a large enough chain-formula then these chain-formula can be lengthened.
Moreover, obtained formula is still a chain-formula that realizes an essential predicate.
Hence, there exists an essential predicate of arbitrarily large arity and the set of predicates does not admit
a near-unanimity function.

%% file: mainResults.tex
\section{Main results}

Let $\mathbb N = \{1,2,3,\ldots\},$
$\mathbb N_{0} = \{0\}\cup \mathbb N,$
$E_{k}=\{0,1,2,\ldots,k-1\}.$
Let $P_{k}^{n} = \{f|\; f: E_{k}^{n}\rightarrow E_{k}\}$ for $n\in \mathbb N,$
and let $P_{k} = \bigcup \limits_{n\ge 1} P_{k}^{n}.$

Suppose $F\subseteq P_{k},$ then by $[F]$ we denote the closure of $F$ under superposition \cite{lau}.
A set $F\subseteq P_{k}$ is called a \emph{clone} if
$F$ is closed and $F$ contains all projections.
By $J_{k}$ we denote the set of all projections.

A mapping $E_{k}^{n}\rightarrow \{0,1\}$ is called an \emph{$n$-ary predicate}.
For $n\in \mathbb N_{0}$ let
$$R_{k}^{n} = \{\rho|\; \rho: E_{k}^{n}\rightarrow \{0,1\}\},$$
$$R_{k} = \bigcup \limits_{n\ge 0} R_{k}^{n}.$$
We do not distinguish sharply between predicates and relations. So instead of
$\rho(a_{1},\ldots,a_{n}) = 1$  we also write $(a_{1},\ldots,a_{n})\in \rho.$
Sometimes we write
$a_{1}a_{2}\ldots a_{h}$ instead of
$(a_{1},a_{2},\ldots,a_{h})$
and operate with tuples like with words.
Suppose $\alpha\in E_{k}^{h},$ then by $\alpha(i)$ we denote
$i$-th element of $\alpha.$
We suppose that functions from $P_{k}$ are also defined in the usual way on the tuples or words
from $E_{k}^{h}.$
That is, suppose $\alpha_{1},\ldots,\alpha_{n} \in E_{k}^{h},$ $f\in P_{k}^{n},$
then $f(\alpha_{1},\ldots,\alpha_{n}) = \beta,$
where $\beta \in E_{k}^{h},$ $\beta(i) = f(\alpha_{1}(i),\alpha_{2}(i),\ldots,\alpha_{n}(i))$
for every $i.$


We say that a function $f\in P_{k}^{m}$ \emph{preserves a predicate} $\rho$
if
$f(\alpha_{1},\alpha_{2},\ldots,\alpha_{m}) \in \rho$
for every $\alpha_{1},\alpha_{2},\ldots,\alpha_{m}\in \rho.$
We say that a set $S\subseteq R_{k}$ \emph{admits a function} $f\in P_{k}^{m}$
if $f$ preserves every predicate from $S.$

By $\Pol(\rho)$ we denote the set of all functions $f\in P_{k}$
that preserve predicate~$\rho.$ For $S\subseteq R_{k}$ we put
$$\Pol(S) = \bigcap \limits_{\rho\in S} \Pol (\rho).$$
By $\Inv(f)$ we denote the set of all predicates $\rho\in R_{k}$
that are preserved by function $f.$ For $M\subseteq P_{k}$ we put
$$\Inv(M) = \bigcap \limits_{f\in M} \Inv (f).$$


A function $f\in P_{k}^{n}$ is called a \emph{near-unanimity function}
iff the following condition holds:
$$\forall a,b \in E_{k} \;\; f(a,b,b,\ldots,b) =
f(b,a,b,\ldots,b) = \ldots = f(b,b,\ldots,b, a) = b.$$
By $NUF_{k}^{n}$ we denote the set of all near-unanimity functions from $P_{k}$ of arity~$n.$
Let $NUF_{k} = \bigcup \limits_{n\ge 3} NUF_{k}^{n}.$

In this paper we consider the following problem.

\textbf{NUF-Problem:} Given a finite set $G\subseteq R_{k}$, decide whether there exists a
near-unanimity function $f\in P_{k}$ such that $f\in \Pol(G).$

By $ar(\rho)$ we denote the arity of predicate $\rho\in R_{k}.$
By $ar(G)$ we denote the maximal arity of predicates from $G\subseteq R_{k}.$
If $|G| = \infty,$ then we put $ar(G) = \infty.$

\begin{theorem}\label{maintheorem}

Suppose $G\subseteq R_{k},$ $|G|<\infty,$
then 
$$NUF_{k}\cap \Pol(G) \neq \varnothing \Longleftrightarrow NUF_{k}^{n}\cap \Pol(G)\neq \varnothing,$$
where $n = (k \cdot ar(G))^{2^{2k^{2}}+2}.$

\end{theorem}

Hence, we obtain the following

\begin{corollary}

\textbf{NUF-Problem} is decidable.

\end{corollary}

%% file: MainStatements.tex
\section{Main statements}

In this section we formulate several statements and prove the main theorem of this paper.

By $\sigma_{k}^{=}$ we denote the predicate from $R_{k}$ such that
$$\sigma_{k}^{=}(x,y) = 1 \Longleftrightarrow x = y.$$
By $false$ we denote the predicate of arity 0 that takes value 0.
Let us give a short definition of the closure operator $[\;]$ on the set $R_{k}.$
You can find a rigorous definition in monograph \cite{lau}.
Suppose $S\subseteq R_{k},$ then by $[S]$ we denote the set of all
predicates $\rho\in R_{k}$ that can be presented by a formula over the set $S\cup \{\sigma_{k}^{=}, false\}$
with only propositional functor $\wedge$ and existential quantifier. That is
$$\rho (x_{1},\ldots,x_{n}) = \exists
y_{1}\ldots \exists y_{l} \;
\rho_{1}(z_{1,1},\ldots,z_{1,n_{1}})\wedge \ldots \wedge
\rho_{s}(z_{s,1},\ldots,z_{s,n_{s}}),$$
where $\rho_{1},\ldots,\rho_{s}\in S\cup \{\sigma_{k}^{=}, false\},$
$z_{i,j}\in \{x_{1},\ldots,x_{n},y_{1},\ldots,y_{l}\}.$


\begin{theorem}\cite{bond,lau}\label{bondar1}
Suppose $J_{k}\subseteq M\subseteq P_{k},$ $S\subseteq R_{k},$
then $[M] = \Pol(\Inv(M)),$
$[S] = \Inv(\Pol(S)).$

\end{theorem}

\begin{theorem}\cite{bond,lau} \label{bondar2}
Let $\mathbb L(P_{k})$ be the set of all clones of $P_{k},$
$\mathbb L(R_{k})$ be the set of all closed subsets of $R_{k}.$
Then $Pol(S)\in \mathbb L(P_{k})$ for every $S\subseteq R_{k},$
$Inv(M)\in \mathbb L(R_{k})$ for every $M\subseteq P_{k}.$
Moreover
$$\Inv: \mathbb L(P_{k}) \longrightarrow \mathbb L(R_{k}),$$
$$\Pol: \mathbb L(R_{k}) \longrightarrow \mathbb L(P_{k})$$
are bijective mappings, which reverse the partial order $\subseteq,$
$i.\,e.,$ it holds
$$\forall A,B\in \mathbb L(P_{k}): A\subseteq B \Rightarrow \Inv(B)\subseteq \Inv(A),$$
$$\forall S,T\in \mathbb L(R_{k}): S\subseteq T \Rightarrow \Pol(T)\subseteq \Pol(S).$$

\end{theorem}

So we have a one-to-one correspondence (which is called the Galois connection)
between closed sets of predicates of $R_{k}$ and clones in $P_{k}$.

Suppose $S\subseteq R_{k},$ then by $\AND(S)$ we denote the set of all $\rho\in R_{k}$ that can be presented by
a formula of the following form:
$$\rho (x_{1},\ldots,x_{n}) = \rho_{1}(z_{1,1},\ldots,z_{1,n_{1}})\wedge \ldots \wedge
\rho_{s}(z_{s,1},\ldots,z_{s,n_{s}}),$$
where $\rho_{1},\ldots,\rho_{s}\in S,$
$z_{i,j}\in \{x_{1},\ldots,x_{n}\},$ $z_{i,j}\neq z_{i,l}$ for all $i,j,l,j\neq l.$

A predicate $\rho$ of arity $n$ is called \textit{essential} if
there do not exist predicates $\rho_{1},\rho_{2},\ldots,\rho_{l}$ such that
$ar(\rho_{i})<n$
for every $i$ and
$\rho \in \AND(\{\rho_{1},\rho_{2},\ldots,\rho_{l}\}).$
The set of all essential predicates of arity $n$ is denoted by $\widetilde R_{k}^{n}.$
Let $\widetilde R_{k} = \bigcup \limits_{n\ge 0} \widetilde R_{k}^{n}.$
This notion was introduced before by the author in \cite{mydm,mybook,Minimal_Clones}.
Using this notion the
lattice of all clones of self-dual functions in three-valued logic
was completely described \cite{SD_DAN, SD_ISMVL,mydm,mybook},
and for every minimal clone in three-valued logic the cardinality of the set of all clones containing
this minimal clone was found \cite{Minimal_Clones}.

The following theorem is proved in Section 5.

\newcounter{nufcounter}
\setcounter{nufcounter}{\value{theorem}}

\begin{theorem}\label{NUFmain}

Suppose $G\subseteq R_{k},$ then
$$\Pol(G)\cap NUF_{k}^{n+1}\neq \varnothing \Longleftrightarrow
ar([G] \cap \widetilde R_{k}) \le n .$$

\end{theorem}

Note that similar theorem but without the notion of an essential predicate
was already proved in \cite{pixley,jeavons}.

\begin{corollary}\label{NUFmainCo}

Suppose $G\subseteq R_{k},$ then
$$\Pol(G)\cap NUF_{k}\neq \varnothing \Longleftrightarrow
|[G] \cap \widetilde R_{k}|<\infty.$$

\end{corollary}

Hence, \textbf{NUF-Problem} is equivalent to the following problem:
given a finite set $G\subseteq R_{k}$, decide whether
the set $[G] \cap \widetilde R_{k}$ is finite.

In our opinion this reformulation of \textbf{NUF-Problem} is even more natural.

By $\rho_{=,a}$ we denote the predicate of arity one defined by the following condition
$$\rho_{=,a}(x) = 1 \Longleftrightarrow x = a.$$

Let $SR_{k}$ be the set of all such predicates in $R_{k}.$
Note that if
$$\sigma(x_{1},\ldots,x_{i-1},x_{i+1},\ldots,x_{n}) =
\rho(x_{1},\ldots,x_{i-1},c, x_{i+1},\ldots,x_{n}),$$
where $c \in E_{k},$
then $\sigma \in [\{\rho\}\cup SR_{k}].$

Every near-unanimity function preserves every predicate form $SR_{k}.$
Hence, by Theorem \ref{NUFmain} we have
\begin{multline*}
ar([G] \cap \widetilde R_{k})\le n \Longleftrightarrow
\Pol(G)\cap NUF_{k}^{n+1}\neq \varnothing \Longleftrightarrow \\
\Longleftrightarrow
\Pol(G\cup SR_{k})\cap NUF_{k}^{n+1}\neq \varnothing \Longleftrightarrow
ar([G\cup SR_{k}] \cap \widetilde R_{k})\le n,
\end{multline*}
$$|[G] \cap \widetilde R_{k}|<\infty \Longleftrightarrow
|[G\cup SR_{k}] \cap \widetilde R_{k}|<\infty.$$

\newcounter{dveteoremy}
\setcounter{dveteoremy}{\value{theorem}}

The following two theorems are proved in Section 8.

\begin{theorem}\label{posledovatelnostPredikatov}

Suppose $SR_{k}\subseteq G,$
$|[G]\cap \widetilde R_{k}|<\infty,$
$ar([G]\cap \widetilde R_{k})=p,$
$ar(G)=q,$
then there exist $\rho \in [G]\cap \widetilde R_{k},$
$\rho_{1},\rho_{2},\ldots,\rho_{n}\in [G]$
such that
\begin{multline*}\rho(x_{1},\ldots,x_{n}) = \exists y_{1}\exists y_{2}\ldots\exists y_{n-1}\;
\rho_{1}(x_{1},y_{1})\wedge
\rho_{2}(y_{1},x_{2},y_{2})\wedge \\ \wedge
\rho_{3}(y_{2},x_{3},y_{3})\wedge \ldots \wedge
\rho_{n-1}(y_{n-2},x_{n-1},y_{n-1})\wedge
\rho_{n}(y_{n-1},x_{n})
\end{multline*}
and
$ar(\rho)>\log_{k\cdot q}(p).$

\end{theorem}

\begin{theorem}\label{beskonechnayaPosledovatelnost}

Suppose $\rho \in \widetilde R_{k},$ $G\subseteq R_{k},$ $\rho_{1},\rho_{2},\ldots,\rho_{n}\in [G]$
\begin{multline*}\rho(x_{1},\ldots,x_{n}) = \exists y_{1}\exists y_{2}\ldots\exists y_{n-1}\;
\rho_{1}(x_{1},y_{1})\wedge
\rho_{2}(y_{1},x_{2},y_{2})\wedge \\ \wedge
\rho_{3}(y_{3},x_{3},y_{3})\wedge \ldots \wedge
\rho_{n-1}(y_{n-2},x_{n-1},y_{n-1})\wedge
\rho_{n}(y_{n-1},x_{n})
\end{multline*}
where $n>2^{2k^{2}}+2.$
Then $|[G] \cap \widetilde R_{k}| = \infty.$

\end{theorem}

Let us prove the main theorem from the previous section.
\begin{Proof} [ of Theorem \ref{maintheorem}]

Assume that
$$NUF_{k}\cap \Pol(G)\neq \varnothing,\;\;
NUF_{k}^{n}\cap \Pol(G)= \varnothing,$$
where $n = (k\cdot ar(G))^{2^{2k^{2}}+2}.$
Hence
$$NUF_{k}\cap \Pol(G\cup SR_{k})\neq \varnothing,\;\;
NUF_{k}^{n}\cap \Pol(G\cup SR_{k})= \varnothing.$$
By Theorem \ref{NUFmain},
we have
$
n \le ar([G\cup SR_{k}]\cap \widetilde R_{k})<\infty.$
Let $ar([G\cup SR_{k}]\cap \widetilde R_{k})=p,$
$ar(G\cup SR_{k})=q.$
By Theorem~\ref{posledovatelnostPredikatov},
there exist $\rho \in [G\cup SR_{k}]\cap \widetilde R_{k},$
$\rho_{1},\rho_{2},\ldots,\rho_{n}\in [G\cup SR_{k}]$
such that
\begin{multline*}\rho(x_{1},\ldots,x_{n}) = \exists y_{1}\exists y_{2}\ldots\exists y_{n-1}\;
\rho_{1}(x_{1},y_{1})\wedge
\rho_{2}(y_{1},x_{2},y_{2})\wedge \\ \wedge
\rho_{3}(y_{2},x_{3},y_{3})\wedge \ldots \wedge
\rho_{n-1}(y_{n-2},x_{n-1},y_{n-1})\wedge
\rho_{n}(y_{n-1},x_{n})
\end{multline*}
and
$ar(\rho)>\log_{k\cdot q}(p).$
Hence
$$ar(\rho)>\log_{k\cdot q}(p)\ge \log_{k\cdot q}(n)=
\log_{k\cdot q}((k\cdot q)^{2^{2k^{2}}+2}) = 2^{2k^{2}}+2.$$

By Theorem~\ref{beskonechnayaPosledovatelnost},
we have $|[G\cup SR_{k}] \cap \widetilde R_{k}| = \infty.$
This contradiction concludes the proof.
%
%

\end{Proof} 

%% file: sush.tex
\section{Essential predicates}%

Note that some statements from this section were already proved in \cite{mydm, mybook,Minimal_Clones}.

In some cases we use the notation $\rho(\ldots).$
This means that $\rho$ depends on some variables but the exact list of variables is omitted.
Usually this list is not important or can be found from the context.

We say that $i$-th variable of a predicate $\rho\in R_{k}^{n}$ is \emph{essential} if
there exist $a_{1},a_{2},\ldots,a_{n},b\in E_{k}$
such that
$$\rho(a_{1},a_{2},\ldots,a_{i-1},a_{i},a_{i+1},\ldots, a_{n}) \neq
\rho(a_{1},a_{2},\ldots,a_{i-1},b,a_{i+1},\ldots, a_{n}).$$


Suppose $\rho \in R_{k},$ then by $\Strike(\rho)$ we denote the set of all $\rho'$ that
can be presented by
a formula of the following form:
$$\rho' (x_{1},\ldots,x_{n}) = \exists y_{1} \exists y_{2}\ldots  \exists y_{l} \; \rho(z_{1},\ldots,z_{m})$$
where $l\ge 0,$ $z_{1},z_{2},\ldots,z_{m} \in \{x_{1},\ldots,x_{n},y_{1},\ldots,y_{l}\},$
$z_{i}\neq z_{j}$ if $i\neq j.$


Suppose $\rho_{1},\rho_{2}\in R_{k}^{n};$ we say that $\rho_{1}\le \rho_{2}$
if $\rho_{1}(a_{1},\ldots,a_{n})\le \rho_{2}(a_{1},\ldots,a_{n})$ for every
$a_{1},\ldots,a_{n}\in E_{k};$
we say that $\rho_{1}<\rho_{2}$ if $\rho_{1}\le \rho_{2}$ and $\rho_{1}\neq \rho_{2}.$

A tuple $(a_{1},a_{2},\ldots,a_{n})$ is called \textit{essential for a predicate $\rho\in R_{k}^{n}$}
if $$\rho(a_{1},a_{2},\ldots,a_{n})=0$$ and
there exist $b_{1},b_{2},\ldots,b_{n}\in E_{k}$ such that for every $i\in \{1,2,\ldots,n\}$
$$\rho(a_{1},\ldots,a_{i-1},b_{i},a_{i+1},\ldots,a_{n})=1.$$

Let us  define the predicate $\widetilde \rho$
for every $\rho\in R_{k}^{n},$  where $n\ge 1.$
Let 
$$\sigma_{i}(x_{1},\ldots,x_{i-1},x_{i+1},\ldots,x_{n}) = \exists x_{i}\; \rho(x_{1},\ldots,x_{n}).$$
By $\widetilde \rho$ we denote the following predicate:
$$\widetilde \rho(x_{1},\ldots,x_{n}) =
\sigma_{1}(x_{2},\ldots,x_{n})
\wedge\sigma_{2}(x_{1},x_{3},\ldots,x_{n})
\wedge \ldots\wedge \sigma_{n}(x_{1},\ldots,x_{n-1}).$$




\begin{lemma}\label{sushnabor}

Suppose $\rho \in R_{k}^{n},$ where $n\ge 1.$ Then the following conditions are equivalent:

1) $\rho$ is an essential predicate;

2) $\rho \neq \widetilde \rho;$

3) there exists an essential tuple for $\rho$.

\end{lemma}

\begin{proof}

Let $\sigma_{i}(x_{1},\ldots,x_{i-1},x_{i+1},\ldots,x_{n}) = \exists x_{i}\; \rho(x_{1},\ldots,x_{n}).$
Then $$\widetilde \rho (x_{1},\ldots,x_{n}) = \sigma_{1}(x_{2},\ldots,x_{n})\wedge  \sigma_{2}(x_{1},x_{3},\ldots,x_{n})\wedge
\ldots\wedge \sigma_{n}(x_{1},\ldots,x_{n-1}).$$
Let us prove that the first condition implies the second condition,
the second implies the third and the third implies the first.

Suppose $\rho$ is essential, then it follows from the definition that
$\rho \neq \widetilde\rho.$

Suppose  $\rho \neq \widetilde\rho.$ It can be easily checked that $\rho \le \widetilde \rho.$
Then there exists $(a_{1},\ldots,a_{n})$
such that $\widetilde \rho(a_{1},\ldots,a_{n})=1,$
$\rho(a_{1},\ldots,a_{n})=0.$
By definition of the predicates
$\sigma_{1},\ldots,\sigma_{n},$
for every $i$ there exists
$b_{i}\in E_{k}$
such that
$$\rho(a_{1},\ldots,a_{i-1},b_{i},a_{i+1},\ldots,a_{n}) = 1.$$
Hence the tuple $(a_{1},\ldots,a_{n})$ is an essential tuple for $\rho.$

Suppose $(a_{1},\ldots,a_{n})$ is an essential tuple for $\rho.$
Assume that $\rho$ is not essential. Then there exist $\rho_{1},\ldots, \rho_{l}\in R_{k}$
such that
$$\rho(x_{1},\ldots,x_{n}) = \rho_{1}(\ldots)\wedge \ldots\wedge \rho_{l}(\ldots)$$
and
$ar(\rho_{i})<n$
for every $i.$
Without loss of generality it can be assumed that
every predicate $\rho_{i}$ depends on all variables $x_{1},\ldots,x_{n}$,
but at least one of these variables is not essential in $\rho_{i}.$
Since $\rho(a_{1},\ldots,a_{n})=0,$
there exists $j$ and $i$ such that $\rho_{j}(a_{1},\ldots,a_{n}) =0$
and the $i$-th variable of $\rho_{j}$ is not essential.
Hence there is no $b_{i}$ such that
$\rho_{j}(a_{1},\ldots,a_{i-1},b_{i},a_{i+1},\ldots,a_{n}) =1.$
Therefore $(a_{1},a_{2},\ldots,a_{n})$ is not an essential tuple.

\end{proof}

\begin{lemma}\label{razlojenieNesushestvennogo}


Suppose $\rho\in R_{k},$   then
$\rho \in \AND(\Strike(\rho)\cap \widetilde R_{k}).$


\end{lemma}

\begin{proof}

The proof is by induction on the arity of $\rho.$
If the arity of $\rho$ is equal to~0, then $\rho$ is essential and the proof is trivial.
If $\rho$ is an essential predicate, then the lemma is trivial.
Suppose $\rho$ is not essential.
Then by Lemma \ref{sushnabor}, it follows that
$$\rho (x_{1},x_{2},\ldots,x_{n}) = \sigma_{1}(x_{2},\ldots,x_{n})\wedge  \sigma_{2}(x_{1},x_{3},\ldots,x_{n})\wedge
\ldots\wedge \sigma_{n}(x_{1},\ldots,x_{n-1}),$$
where $\sigma_{i}(x_{1},\ldots,x_{i-1},x_{i+1},\ldots,x_{n}) =
\exists x_{i}\; \rho(x_{1},\ldots,x_{n}).$
By the inductive assumption,
$\sigma_{i} \in  \AND(\Strike(\sigma_{i})\cap \widetilde R_{k}).$
Hence
$$\rho \in \AND\left(\bigcup \limits_{i=1}^{n} \AND\left(\Strike(\sigma_{i})\cap \widetilde R_{k}\right)\right)
\subseteq  \AND\left( \left(\bigcup \limits_{i=1}^{n} \Strike(\sigma_{i})\right) \cap \widetilde R_{k}\right).$$
Since
$\Strike(\sigma_{i})\subset \Strike(\rho)$ for every $i$,
we have $\rho \in \AND(\Strike(\rho)\cap \widetilde R_{k}).$

\end{proof}

\begin{lemma}\label{SuzhZamkn}

$[[S]\cap \widetilde R_{k}] = [S]$ for every $S\subseteq R_{k}.$

\end{lemma}

\begin{proof}

The inclusion $[[S]\cap \widetilde R_{k}] \subseteq [S]$ is trivial.
Let us prove the inclusion $[[S]\cap \widetilde R_{k}] \supseteq [S].$
Suppose $\rho \in [S],$ then by Lemma \ref{razlojenieNesushestvennogo},
it follows that $\rho \in \AND(\Strike(\rho)\cap \widetilde R_{k}).$
Since $\Strike(\rho)\subseteq [\{\rho\}]$ and $\AND(T)\subseteq [T]$ for every
$T\subseteq R_{k},$ we get
$\rho \in [[\{\rho\}]\cap \widetilde R_{k}]\subseteq [[S]\cap \widetilde R_{k}].$
This concludes the proof.

\end{proof}

It follows from the previous lemma that every closed set $S\subseteq R_{k}$
can be described by the set $S\cap \widetilde R_{k}$ of all essential predicates of $S.$

\begin{lemma}\label{arnostSush}

Suppose $\rho\in R_{k},$ $\alpha_{i},\beta_{i}\in E_{k}^{s_{i}}$ for $i\in \{1,2,\ldots,n\},$
$$\rho(\alpha_{1}\ldots\alpha_{n}) = 0,$$
$$\forall j \; \rho(\alpha_{1}\ldots\alpha_{j-1}\beta_{j}\alpha_{j+1}\ldots\alpha_{n}) =1,$$
then there exists an essential predicate $\rho' \in [\{\rho\}]$
such that $ar(\rho') \ge n.$

\end{lemma}

\begin{proof}

Let us prove this by induction on the arity of $\rho.$
Since $\sigma_{k}^{=}\in [\{\rho\}],$ the lemma is trivial for $ar(\rho)\le 2$.
Let $\gamma = \alpha_{1}\alpha_{2}\ldots\alpha_{n}.$
If $\rho$ is an essential predicate, then the proof is trivial. 
Assume that $\rho$ is not essential and $ar(\rho) = m.$
Suppose $\gamma_{i}$ is obtained from $\gamma$ by removing $i$-th element.
Let $$\rho_{i}(x_{1},\ldots,x_{i-1},x_{i+1},\ldots,x_{m}) = \exists x_{i} \; \rho(x_{1},\ldots,x_{m}).$$
By Lemma \ref{sushnabor},
$\rho = \widetilde\rho.$
Hence, $\rho(\gamma) = \rho_{1}(\gamma_{1})\wedge \ldots
\wedge \rho_{m}(\gamma_{m})=0,$
and $\rho_{i}(\gamma_{i}) = 0$ for some $i.$
Without loss of generality it can be assumed that
$i=1.$
Since $\rho(\beta_{1}\alpha_{2}\ldots\alpha_{n}) = 1,$
the length of $\alpha_{1}$ is greater then one.
Suppose $\alpha_{1}'$ is obtained from $\alpha_{1}$ by removing the first element,
and $\beta_{1}'$ is obtained from $\beta_{1}$ by removing the first element.
Therefore, we have
$$\rho_{1}(\alpha_{1}'\alpha_{2}\alpha_{3}\ldots\alpha_{n}) = 0,$$
$$\rho_{1}(\beta_{1}'\alpha_{2}\alpha_{3}\ldots\alpha_{n}) =1,$$
$$\forall j\ge 2 \;\; \rho_{1}(\alpha_{1}'\alpha_{2}\ldots\alpha_{j-1}\beta_{j}\alpha_{j+1}\ldots\alpha_{n}) =1.$$
Hence, by the inductive assumption
there exists an essential predicate $\rho' \in [\{\rho_{1}\}]\subseteq [\{\rho\}]$
such that $ar(\rho') \ge n.$ This completes the proof.

%
%

\end{proof}

\begin{lemma}\label{nebolshek}

Suppose $\rho,\rho_{1},\ldots,\rho_{m}\in R_{k},$ $m>k$ and
\begin{multline*}
\rho(x_{1,1},\ldots,x_{1,n_{1}},x_{2,1},\ldots,x_{2,n_{2}},\ldots,x_{m,1},\ldots,x_{m,n_{m}}) = \\ =
\exists y \; \rho_{1}(y,x_{1,1},\ldots,x_{1,n_{1}})\wedge
\ldots \wedge \rho_{m}(y,x_{m,1},\ldots,x_{m,n_{m}}).
\end{multline*}
Then $\rho$ is not an essential predicate.

\end{lemma}

\begin{proof}

Assume the converse. By Lemma \ref{sushnabor},
there exists an essential tuple $\gamma$ for $\rho.$
Suppose
$\gamma = \alpha_{1}\alpha_{2}\ldots\alpha_{m}$
where $\alpha_{i}\in E_{k}^{n_{i}}$ for every $i.$
Put $C_{i} = \{c\in E_{k}\;|\; \rho_{i}(c\alpha_{i})\} = 1.$
Since $\gamma$ is an essential tuple,
we have $$C_{1}\cap C_{2}\cap \ldots \cap C_{m} = \varnothing,$$
$$D_{j} = \bigcap \limits_{i\neq j} C_{i} \neq \varnothing.$$
Hence $D_{i}\cap D_{j} = \varnothing$ for every $i,j.$
Since $D_{i}\subseteq E_{k}$ for every $i$, we have $m\le k.$
This concludes the proof.

\end{proof}

Suppose $G\subseteq R_{k},$ $M = \{\alpha_{1},\ldots,\alpha_{n}\}\subseteq E_{k}^{h}.$
Let us define a predicate $\rho_{M,G}.$
Put
$$\rho_{M,G} = \{ f(\alpha_{1},\alpha_{2},\ldots, \alpha_{n})\; |\;
f\in Pol(G)\cap P_{k}^{n}\}.$$

\begin{lemma}\label{trivialG}

Suppose $G\subseteq R_{k},$
$M = \{\alpha_{1},\ldots,\alpha_{n}\}\subseteq E_{k}^{h}.$
Then
$\rho_{M,G} \in [G].$

\end{lemma}

\begin{proof}

Assume the converse. By Theorem~\ref{bondar1},
we have 
$[G] = \Inv(\Pol(G)).$
Hence, $\rho_{M,G}\notin \Inv(\Pol(G))$ and there exists a function
$f\in \Pol(G)$ such that
$f\notin \Pol(\rho_{M,G}).$
Therefore,
$f(\beta_{1},\ldots,\beta_{s}) \notin \rho_{M,G}$
for some $\beta_{1},\ldots,\beta_{s} \in \rho_{M,G}.$
By definition of $\rho_{M,G}$ there exist functions
$f_{1},\ldots,f_{s}\in \Pol(G)$ such that
$f_{i}(\alpha_{1},\alpha_{2},\ldots, \alpha_{n}) = \beta_{i}$ for every $i.$
Put $$g(x_{1},\ldots,x_{n}) = f(f_{1}(x_{1},\ldots,x_{n}),
f_{2}(x_{1},\ldots,x_{n}),\ldots,
f_{s}(x_{1},\ldots,x_{n})).$$
Then $g\in \Pol(G)$ and $$g(\alpha_{1},\alpha_{2},\ldots, \alpha_{n})=
f(\beta_{1},\ldots,\beta_{s})\in \rho_{M,G}.$$
This contradiction concludes the proof.

\end{proof}


Let us prove a theorem from Section 2.

\newcounter{backupone}
\setcounter{backupone}{\value{theorem}}
\setcounter{theorem}{\value{nufcounter}}

\begin{theorem}

Suppose $G\subseteq R_{k},$ then
$$Pol(G)\cap NUF_{k}^{n+1}\neq \varnothing \Longleftrightarrow
ar([G] \cap \widetilde R_{k}) \le n .$$

\end{theorem}

\setcounter{theorem}{\value{backupone}}

\begin{proof}

Suppose that $Pol(G)\cap NUF_{k}^{n+1}\neq  \varnothing.$
Then there exists $h\in Pol(G)\cap NUF_{k}^{n+1}.$
Assume that $ar([G] \cap \widetilde R_{k})> n .$
Then we have $\rho \in [G] \cap \widetilde R_{k}^{m}$
where $m>n.$
Suppose $(a_{1},\ldots,a_{m})$ is an essential tuple for $\rho,$
then there exist $b_{1},\ldots,b_{m}\in E_{k}$ such that for every $i$
$$\rho(a_{1},\ldots,a_{i-1},b_{i},a_{i+1},\ldots,a_{m}) = 1.$$
Since $h\in Pol(\rho),$ we have
$$h \begin{pmatrix}
b_{1} & a_{1} & a_{1} &\ldots & a_{1} \\
a_{2} & b_{2} & a_{2} &\ldots & a_{2} \\
a_{3} & a_{3} & b_{3} &\ldots & a_{3} \\
\ldots &\ldots &\ldots &\ldots & \ldots \\
a_{n+1} & a_{n+1} & a_{n+1} &\ldots & b_{n+1} \\
a_{n+2} & a_{n+2} & a_{n+2} &\ldots & a_{n+2} \\
\ldots &\ldots &\ldots &\ldots & \ldots \\
a_{m} & a_{m} & a_{m} &\ldots & a_{m}
 \end{pmatrix}=
 \begin{pmatrix}
a_{1}\\a_{2}\\a_{3}\\ \ldots \\ a_{n+1}\\ a_{n+2} \\ \ldots \\ a_{m}
 \end{pmatrix}
 \in \rho.$$
This contradiction proves that $ar([G] \cap \widetilde R_{k}) \le n.$

Suppose that $ar([G] \cap \widetilde R_{k}) \le n.$
Let us define a matrix.
It has $n+1$ columns and $(n+1)\cdot k^{2}$ rows.
The first $k^{2}$ rows contain all
tuples of the form $(a,b,b,\ldots, b),$
where $a,b\in E_{k}.$
The next $k^{2}$ rows contain all
tuples of the form $(b,a,b,\ldots, b),$
where $a,b\in E_{k},$ and so on.
By $\alpha_{i}$ we denote the $i$-th column of this matrix.
Let $M = \{\alpha_{1},\alpha_{2},\ldots,\alpha_{n+1}\}.$

By Lemma \ref{trivialG}, $\rho_{M,G}\in [G].$
It follows from the definition of $\alpha_{1},\alpha_{2},\ldots,\alpha_{n+1}$ that
there exist
$$\gamma_{1},\gamma_{2},\ldots,\gamma_{n+1},
\delta_{1},\delta_{2},\ldots,\delta_{n+1}\in E_{k}^{k^{2}}$$
such that
$$\alpha_{i} =
\gamma_{1}\gamma_{2}\ldots \gamma_{i-1} \delta_{i} \gamma_{i+1}\ldots \gamma_{n+1}$$
for every $i.$

Assume that $\gamma_{1}\gamma_{2}\ldots\gamma_{n+1} \in \rho_{M,G},$
then there exists a function $h\in \Pol(G)$
such that $h(\alpha_{1},\alpha_{2},\ldots,\alpha_{n+1})=\gamma_{1}\gamma_{2}\ldots\gamma_{n+1}.$
By definition, $h\in NUF_{k}^{n+1}.$

Assume that $\gamma_{1}\gamma_{2}\ldots\gamma_{n+1} \notin \rho_{M,G},$
Since $Pol(G)$ contains all projections,
$\alpha_{i}\in \rho_{M,G}$ for every~$i.$
Then by Lemma \ref{arnostSush}
we have $ar([G] \cap \widetilde R_{k}) \ge n+1.$
This contradiction completes the proof.

\end{proof}

Suppose $G\subseteq R_{k},$ $|[G]\cap \widetilde R_{k}|<\infty.$
By $MAX(G)$ we denote the set of all predicates $\rho \in [G]$
such that
$\rho$ is an essential predicate of maximal arity in $[G],$
$\rho \not < \sigma$
for every essential predicate $\sigma \in[G]$ of the maximal arity.

\begin{lemma}\label{VycherkivanieSvobodnih}

Suppose $G\subseteq R_{k},$ $|[G]\cap \widetilde R_{k}|<\infty,$
$\rho\in MAX(G),$
$\rho' \in [G],$
$$\rho(x_{1},\ldots,x_{n}) =
\rho'(\underbrace{x_{1},\ldots,x_{1}}_{m_{1}},
\underbrace{x_{2},\ldots,x_{2}}_{m_{2}},
\ldots,
\underbrace{x_{n},\ldots,x_{n}}_{m_{n}}).$$
Then
there exist $i_{1},\ldots,i_{n}$ such that
$$\rho(x_{1},\ldots,x_{n}) =
\exists t_{1} \exists t_{2}\ldots \exists  t_{l} \;  \rho'(z_{1,1},\ldots,z_{1,m_{1}}, 
\ldots,z_{n,1},\ldots,z_{n,m_{n}}),
$$
where $z_{j,i_{j}} = x_{j}$ for every $j,$
$z_{i,j}$ are different for every $i,j.$
Therefore $\rho \in Strike(\rho').$

\end{lemma}

\begin{proof}

Suppose $ar(\rho') = m.$
Let us prove this lemma by induction on $m-n.$
If $m=n$ then the proof is trivial.

Assume that $m>n.$
Let $(a_{1},a_{2},\ldots,a_{n})$ be an essential tuple for $\rho.$
Put $$\gamma =
\underbrace{a_{1}\ldots a_{1}}_{m_{1}}
\underbrace{a_{2}\ldots a_{2}}_{m_{2}}
\ldots
\underbrace{a_{n}\ldots a_{n}}_{m_{n}}.$$

Since $\rho \in MAX(G)$ and $m>n,$ $\rho'$ is not an essential predicate.
Suppose $\gamma_{i}$ is obtained from $\gamma$ by removing $i$-th element.
Let $$\rho_{i}(y_{1},\ldots,y_{i-1},y_{i+1},\ldots,y_{m}) = \exists y_{i} \; \rho'(y_{1},\ldots,y_{m}).$$
By Lemma \ref{sushnabor},
$\rho' = \widetilde\rho'.$
Hence, $\rho'(\gamma) = \rho_{1}(\gamma_{1})\wedge \ldots
\wedge \rho_{m}(\gamma_{m})=0,$
and $\rho_{i}(\gamma_{i}) = 0$ for some $i.$
Without loss of generality it can be assumed that
$i=1.$
Since $$\rho(\underbrace{b_{1}\ldots b_{1}}_{m_{1}}
\underbrace{a_{2}\ldots a_{2}}_{m_{2}}
\ldots
\underbrace{a_{n}\ldots a_{n}}_{m_{n}}) =
\rho_{1}(\underbrace{b_{1}\ldots b_{1}}_{m_{1}-1}
\underbrace{a_{2}\ldots a_{2}}_{m_{2}}
\ldots
\underbrace{a_{n}\ldots a_{n}}_{m_{n}})= 1,$$
we have $m_{1}>1.$ Let
$$\sigma(x_{1},\ldots,x_{n}) =
\rho_{1}(\underbrace{x_{1},\ldots,x_{1}}_{m_{1}-1},
\underbrace{x_{2},\ldots,x_{2}}_{m_{2}},
\ldots,
\underbrace{x_{n},\ldots,x_{n}}_{m_{n}}).$$
It can be easily checked that
$(a_{1},a_{2},\ldots,a_{n})$ is an essential tuple for $\sigma.$
Moreover, $\sigma\ge \rho.$
Since $\rho \in MAX(G),$  we have $\sigma = \rho.$
Hence by the inductive assumption,
there exist $i_{1},\ldots,i_{n}$ such that
$$\rho(x_{1},\ldots,x_{n}) =
\exists t_{1} \exists t_{2}\ldots \exists  t_{l} \;  \rho_{1}(z_{1,1},\ldots,z_{1,m_{1}-1},
\ldots,z_{n,1},\ldots,z_{n,m_{n}}),$$
where $z_{j,i_{j}} = x_{j}$ for every $j,$
$z_{i,j}$ are different for every $i,j.$
Therefore
$$\rho(x_{1},\ldots,x_{n}) =
\exists y_{1} \exists t_{1} \exists t_{2}\ldots \exists  t_{l} \;
\rho' (y_{1},z_{1,1},\ldots,z_{1,m_{1}-1}, 
\ldots,z_{n,1},\ldots,z_{n,m_{n}}).$$
This completes the proof.

%
%

\end{proof}

%
%
%

%% file: GlavnoeUtverjdenie.tex
\section{Main equation}

Suppose $C\subseteq E_{k}$
then by $(x\in C)$ we denote the predicate $\rho$ of arity one such that
$$\rho(x) = 1 \Longleftrightarrow (x\in C).$$

\begin{lemma}\label{PosledovatelnostMnojestv}

Let
$C\subseteq E_{k},$
$\alpha \in E_{k}^{n-2},$
$\rho \in R_{k}^{n},$
$$\forall d \in C\; \exists \beta \in E_{k}^{n-2} \; \rho(d d \beta) =  1,$$
$$D_{0} = \{a\}\subseteq C,
D_{i+1} = \{e\in C | \exists d \in D_{i}:\; \rho(d e \alpha) =1 \},$$
$$\forall m\ge 1 \; \exists j\ge m \; (D_{j}\neq D_{j+1}).$$
Then $$|[\{\rho,(x\in C)\} \cup SR_{k} ]\cap \widetilde R_{k}| = \infty.$$

\end{lemma}

\begin{proof}

Assume the converse. Suppose
$m$ is the maximal arity of essential predicates from $[\{\rho,(x\in C)\} \cup SR_{k} ].$
Assume that
$D_{i}\subseteq D_{i+1}$ for every $i>m.$
Since $E_{k}$ is finite,  there exists
$m_{0}\ge m$ such that $D_{i} = D_{m_{0}}$
for every $i>m_{0}.$
By the conditions of the lemma,
there exists $i\ge m_{0}$ such that
$D_{i}\neq D_{i+1}.$
This contradiction proves that
there exists $j>m$ such that
$D_{j}\not\subseteq D_{j+1}.$

Suppose
$$\sigma_{1}(x_{1},x_{2},y_{1},\ldots, y_{n-2}) = \rho(x_{1},x_{2},y_{1},\ldots, y_{n-2}),$$
\begin{multline*}\sigma_{i+1}(x_{1},x_{2},y_{1},\ldots, y_{(n-2)\cdot (i+1)}) = \\ =
\exists z \; (z\in C)\wedge \sigma_{i}(x_{1},z,y_{1},\ldots, y_{(n-2)\cdot i})
\wedge \rho ( z,x_{2}, y_{(n-2)\cdot i+1},\ldots, y_{(n-2)\cdot (i+1)})\end{multline*}

Let $d \in D_{j}\setminus D_{j+1},$
$$\delta(y_{1},y_{2},\ldots, y_{(n-2)\cdot (j+1)}) =
\sigma_{j+1}(a,d,y_{1},y_{2},\ldots, y_{(n-2)\cdot (j+1)}).$$
It follows from the definition of $D_{j}$
that there exists a sequence
$a_{0},a_{1},\ldots,a_{j}$
such that $a_{0}= a,$ $a_{j} = d,$
$\rho(a_{i}a_{i+1}\alpha) = 1$ for every $i.$

For every $i\in \{0,1,\ldots,j\}$ there is
$\beta_{i} \in E_{k}^{n-2}$ such that
$\rho(a_{i}a_{i}\beta_{i}) = 1.$
Then it is easy to check that 
$\delta(\underbrace{\alpha \ldots \alpha}_{i} \beta_{i} \underbrace{\alpha \ldots \alpha}_{j-i}) =1$
for every $i.$
Since $d\notin D_{j+1},$ we have 
$\delta(\underbrace{\alpha \ldots \alpha}_{j+1}) = 0.$
It follows from Lemma \ref{arnostSush} that
there exists an essential predicate $\rho' \in [\{\delta\}] \subseteq [\{\rho,(x\in C)\} \cup SR_{k} ]$
such that $ar(\rho') \ge j+1>m.$
This contradiction proves the lemma.

\end{proof}

\begin{lemma}\label{EstSymmetrichnoeRebro}

Suppose
$C\subseteq E_{k},$
$\alpha \in E_{k}^{n-2},$
$\rho \in R_{k}^{n},$
$a,b\in C$
$$\forall d \in C \; \rho(d d \alpha) =  0,$$
$$\forall d \in C\; \exists \beta \in E_{k}^{n-2} \; \rho(d d \beta) =  1,$$
$$\rho(a b \alpha) = \rho(b a \alpha) = 1.$$
Then $$|[\{\rho,(x\in C)\} \cup SR_{k} ]\cap \widetilde R_{k}| = \infty.$$
\end{lemma}

\begin{proof}

Assume the converse.
Let $C$ be a minimal set such that
conditions of the lemma hold but
$|[\{\rho,(x\in C)\} \cup SR_{k} ]\cap \widetilde R_{k}|<\infty.$
Suppose
$$D_{0} = \{a\},
D_{i+1} = \{e\in C | \exists d \in D_{i}\; \rho(d e \alpha) \wedge \rho(e d \alpha) =1 \}.$$
It follows from the conditions that
$a\in D_{i}$ for even $i,$
$b\in D_{i}$ for odd $i.$
Hence $D_{i} \neq \varnothing $ for every~$i.$

Let $\rho_{0}(x) = 1 \Longleftrightarrow x = a.$ For $i\ge 0$ suppose
$$\rho_{i+1}(x) = \exists z \; (x\in C)\wedge \rho_{i}(z)\wedge \rho ( x z \alpha) \wedge \rho ( z x \alpha),$$
It is easy to check that  for every $i$
$$\rho_{i}(x) = 1 \Longleftrightarrow x \in D_{i}.$$

Let us prove by induction that $D_{i}\cap D_{i+1} = \varnothing$ for every $i.$
Since $a \notin D_{1},$ this is true for $i=0.$
Let us prove this for $i\ge 1.$
Assume the converse.
Suppose
$D_{i}\cap D_{i+1}\neq \varnothing.$
Let $c\in D_{i}\cap D_{i+1}.$
By the inductive assumption,
$D_{i-1}\cap D_{i}=\varnothing.$
Since $D_{i-1}\neq \varnothing,$ then $D_{i} \neq C.$

Since $c\in D_{i+1},$ there exists $d\in D_{i},$
such that $\rho(d c \alpha) = \rho(c d\alpha) = 1.$
By the assumption,
$C$ is a minimal set such that
conditions of the lemma hold but
$|[\{\rho,(x\in C)\} \cup SR_{k} ]\cap \widetilde R_{k}|<\infty.$
Since $D_{i}\subset C,$ $c,d\in D_{i},$ $\rho(c d \alpha) = \rho(d c \alpha)=1,$
we have $$|[\{\rho,\rho_{i}\} \cup SR_{k} ] \cap \widetilde R_{k}| = \infty.$$
But $$[\{\rho,\rho_{i}\} \cup SR_{k} ] \subseteq [\{\rho,(x\in C)\} \cup SR_{k} ].$$
Hence $|[\{\rho,(x\in C)\} \cup SR_{k} ] \cap \widetilde R_{k}| = \infty$ and the lemma is proved.

Thus, $D_{i}\cap D_{i+1} = \varnothing$ for every $i.$
Hence we can use Lemma \ref{PosledovatelnostMnojestv} for the predicate
$\rho(x_{1},x_{2},x_{3},\ldots,x_{n})\wedge \rho(x_{2},x_{1},x_{3},\ldots,x_{n})$
to complete the proof.

\end{proof}

\begin{lemma}\label{EstCykl}

Suppose
$C\subseteq E_{k},$
$\alpha \in E_{k}^{n-2},$
$\rho \in R_{k}^{n},$
$a_{0},a_{1},\ldots,a_{m}\in C,$ $a_{0} = a_{m},$
$$\forall i\in \{0,1,\ldots,m-1\} \; \rho(a_{i} a_{i+1}\alpha) =  1,$$
$$\forall d \in C \; \rho(d d \alpha) =  0,$$
$$\forall d \in C\; \exists \beta \in E_{k}^{n-2} \; \rho(d d \beta) =  1.$$
Then $$|[\{\rho,(x\in C)\} \cup SR_{k} ]\cap \widetilde R_{k}| = \infty.$$
\end{lemma}

\begin{proof}

Let
$$D_{0} = \{a_{0}\},
D_{i+1} = \{e\in C | \exists d \in D_{i}\; \rho(d e \alpha)  =1 \}.$$
It can be easily checked that $D_{i}\neq \varnothing$ for every $i.$
Assume that for every $m_{0}$ there exists $m'\ge m_{0}$ such that $D_{m'}\neq D_{m'+1};$
then we can use Lemma \ref{PosledovatelnostMnojestv} to complete the proof.

Suppose there exists $m_{0}$ such that $D_{m'} = D_{m_{0}}$  for every $m'>m_{0}.$
It is easy to check that $a_{0}\in D_{m\cdot i}$ for every $i\ge 0.$
Therefore, $a_{0}\in D_{m'}$  for every $m'>m_{0}.$
Let
$$\sigma(x_{1}, x_{2},y_{1},\ldots,y_{n-2}) = \rho(x_{1}, x_{2},y_{1},\ldots,y_{n-2}) \wedge (x_{1}\in C) \wedge 
(x_{2}\in C),$$
$$\sigma_{s}(x_{0},x_{s},y_{1},\ldots,y_{n-2}) =
\exists x_{1} \ldots\exists x_{s-1}\;
\bigwedge \limits_{1\le i\le s}
\sigma(x_{i-1}, x_{i},y_{1},\ldots,y_{n-2}).$$
Roughly speaking, this means that $\sigma_{s}$ takes value 1 iff there exists a path of length $s$ from
$x_{0}$ to $x_{s}.$
Let
$$M_{0} = \{ s\in \mathbb N| \;\exists d\in C:\; \sigma_{s}(d d \alpha) = 1\}.$$
Roughly speaking, $M_{0}$ is the set of all numbers $s$ such that
there exists a cycle of length $s.$
By the condition of the lemma, $1\notin M_{0}.$
Since $a_{0}\in D_{m'}$  for every $m'>m_{0},$
we get
$\sigma_{s}(a_{0} a_{0} \alpha) = 1$
for every $s>m_{0}.$ Then $s\in M_{0}$ for every $s>m_{0}.$
Let $r\ge 0$ be the minimal number such that $2^{r}\in M_{0}.$
Hence, there exist
$b_{0},b_{1},\ldots,b_{2^{r}}\in C$
such that  $b_{0} = b_{2^{r}}$ and for every $i\in \{0,1,\ldots,2^{r}-1\}$
$$\sigma(b_{i} b_{i+1}\alpha) =  1.$$
Therefore, 
$$\sigma_{2^{r-1}}(b_{0},b_{2^{r-1}}) =  \sigma_{2^{r-1}}(b_{2^{r-1}}, b_{0}) =1.$$
Since $r$ is minimal, we have 
$$\forall d \in C \; \sigma_{2^{r-1}}(d d \alpha) =  0.$$
Hence by Lemma \ref{EstSymmetrichnoeRebro},
we have
$$|[\{\sigma_{2^{r-1}},(x\in C)\} \cup SR_{k} ]\cap \widetilde R_{k}| = \infty.$$
This completes the proof.


%
%

\end{proof}

\begin{lemma}\label{dobavleniePredikatovSR}

Suppose $\rho \in R_{k},$ then 
$$|[\{\rho\}]\cap \widetilde R_{k}| < \infty \Longleftrightarrow
|[\{\rho\}\cap SR_{k}]\cap \widetilde R_{k}| < \infty.$$

\end{lemma}

\begin{proof}

By definition, every near unanimity function preserves every predicate form $SR_{k}.$
Hence, by Corollary \ref{NUFmain} we have
\begin{multline*}|[G] \cap \widetilde R_{k}|<\infty \Longleftrightarrow
\Pol(G)\cap NUF_{k}\neq \varnothing \Longleftrightarrow  \\  \Longleftrightarrow
\Pol(G\cup SR_{k})\cap NUF_{k}\neq \varnothing \Longleftrightarrow
|[G\cup SR_{k}] \cap \widetilde R_{k}|<\infty.
\end{multline*}

\end{proof}

\begin{theorem}\label{GlavnoePreobrazovanie}

Suppose $\rho
\in R_{k}^{n+2},$
$\rho_{0}(y) = \exists x_{1}\exists x_{2}\ldots \exists x_{n}\; \rho(y,y,x_{1},x_{2}\ldots,x_{n}),$
$|[\{\rho\}]\cap \widetilde R_{k}| < \infty.$
Then
\begin{multline*}\exists y \; \rho(y,y,x_{1},\ldots,x_{n}) = \\ =
\exists y_{0} \exists y_{1}\ldots
\exists y_{k}\;
(\bigwedge\limits_{0\le i\le k-1}
\rho(y_{i},y_{i+1},x_{1},\ldots,x_{n}))
\wedge
(\bigwedge\limits_{0\le i\le k}
\rho_{0}(y_{i})).
\end{multline*}

\end{theorem}

\begin{proof}

Assume the converse.
By $\rho_{1}(x_{1},\ldots,x_{n})$ we denote the predicate in the left side of the formula from the condition.
By $\rho_{2}(x_{1},\ldots,x_{n})$ we denote the predicate in the right side.
Let us prove that $\rho_{1} = \rho_{2}.$

Let $C = \{ d\in E_{k}|\; \rho_{0}(d) = 1\}.$
It can be easily checked that
$\rho_{1}\le \rho_{2}.$
Assume that $\rho_{2}\neq \rho_{1},$
then there exists $\alpha\in E_{k}^{n}$
such that $\rho_{2}(\alpha) = 1,$
$\rho_{1}(\alpha) = 0.$
Therefore there exist $a_{0},a_{1},\ldots,a_{k}\in C$
such that $\rho(a_{i}a_{i+1}\alpha) = 1$ for every $i\in \{0,1,\ldots,k-1\}.$
Obviously, $a_{i} = a_{j}$ for some $i,j\in \{0,1,\ldots,k\}.$
Since $\rho_{1}(\alpha) = 0,$ we have
$\rho(d d \alpha) = 0$ for every $d\in E_{k}.$
By Lemma \ref{EstCykl}, we have
$|[\{\rho\} \cup SR_{k} ]\cap \widetilde R_{k}| = \infty.$
Hence, by Lemma \ref{dobavleniePredikatovSR},
$|[\{\rho\}]\cap \widetilde R_{k}| = \infty.$
This completes the proof.

\end{proof}

%% file: transformations.tex
\section{Transformations of formulas}%

Suppose $G\subseteq R_{k},$ $|G|<\infty,$ $|[G]\cap \widetilde R_{k}|<\infty,$
$[G]\cap (R_{k}^{1}\cup R_{k}^{2}) \subseteq G.$
In this section, $G$ always satisfies these conditions.

By $Formulas(G)$ we denote the set of all formulas of the following form
$$\Phi = \exists
y_{1}\ldots \exists y_{l} \;
\rho_{1}(z_{1,1},\ldots,z_{1,n_{1}})\wedge \ldots \wedge
\rho_{s}(z_{s,1},\ldots,z_{s,n_{s}}),$$
where $\rho_{1},\ldots,\rho_{s}\in G.$
Variables $y_{1},\ldots,y_{l}$ are called \emph{bound} variables.
All other variables are called \emph{unbound}.

Suppose $\{j_{1},\ldots,j_{r}\}\subseteq \{1,2,\ldots,l\},$
$\{i_{1},\ldots,i_{p}\}\subseteq \{1,2,\ldots,s\}.$
Then the formula
$$\Phi' = \exists
y_{j_{1}}\ldots \exists y_{j_{r}} \;
\rho_{i_{1}}(z_{i_{1},1},\ldots,z_{i_{1},n_{i_{1}}})\wedge \ldots \wedge
\rho_{i_{p}}(z_{i_{p},1},\ldots,z_{i_{p},n_{i_{p}}})$$
is called a subformula of $\Phi.$
The subformula $\rho_{i}(z_{1,1},\ldots,z_{1,n_{1}})$ is called
an \emph{occurrence} of the formula $\Phi.$

We suppose that existential quantifiers are always in the left part of the formula.
Also we suppose that if
$\Phi_{1}$ and $\Phi_{2}$ are formulas
then all bound variables in the formula
$\Phi_{1}\wedge \Phi_{2}$ are different.

Suppose $\Phi\in Formulas(G).$ Then by $Var(\Phi)$ we denote the set of all variables in $\Phi.$
By $UVar(\Phi)$ we denote the set of all unbound variables in $\Phi.$

Suppose 
$m\ge 1,$ $z_{0},z_{1},\ldots,z_{m}\in Var(\Phi),$
$\Psi_{1},\ldots,\Psi_{m}$ are occurrences of formula $\Phi.$
Then the sequence $z_{0}\Psi_{1}z_{1}\Psi_{2}\ldots \Psi_{m}z_{m}$
is called a \emph{path} from $z_{0}$ to $z_{m}$ in $\Phi$
iff the following conditions hold:



\begin{enumerate}

\item $z_{i}\in Var(\Psi_{i+1})$ 
for every $i\in \{0,1,\ldots,m-1\};$

\item $z_{i}\in Var(\Psi_{i})$ for every $i\in \{1,\ldots,m\};$


\item variables $z_{0},z_{1},\ldots,z_{m-1},z_{m}$ are different except may be $z_{0}$ and $z_{m}.$

\end{enumerate}


We say that $m$ is the \emph{length} of the path.

We say that variables $x$ and $y$ are \emph{connected in $\Phi$} if there exists a path in $\Phi$
from $x$ to $y.$
We say that $\Phi$ is \emph{connected} if every two variables of $\Phi$ are connected.

\begin{lemma} \label{svyasnyaformula}

Suppose $\Phi\in Formulas(G),$
$\Phi$ realizes an essential predicate $\rho,$
then there exists a connected subformula $\Phi'$ of $\Phi$
that realizes predicate $\rho.$

\end{lemma}

\begin{proof}

It can be easily checked that if $x$ and $y$ are connected in $\Phi,$
$y$ and $z$ are connected in $\Phi,$ then $x$ and $z$ are connected in $\Phi.$
So, all variables from $Var(\Phi)$ can be divided into equivalence classes
$V_{1}, V_{2}, \ldots, V_{s}.$
It can be shown that $\Phi$ can be divided into $s$ connected subformulas
$\Phi_{1},\Phi_{2},\ldots,\Phi_{s}$
such that
$Var(\Phi_{i}) = V_{i}$ for every $i$ and
$$\Phi = \Phi_{1}\wedge \Phi_{2} \wedge\ldots\wedge\Phi_{s}.$$
Since $\Phi$ realizes an essential predicate,
there exists $i$ such that
$UVar(\Phi) = UVar(\Phi_{i})$
and $UVar(\Phi_{j})=\varnothing$ for $j\neq i.$
Therefore $\Phi_{i}$ realizes the same predicate as $\Phi.$

\end{proof}

By $GF(G)$ we denote the set of all connected formulas $\Phi\in Formulas(G)$ such that
each unbound variable in $\Phi$ is used just once and
$\Phi$ realizes a predicate from $MAX(G).$

\begin{lemma}

The set $GF(G)$ is not empty.

\end{lemma}

\begin{proof}

Let $\rho\in MAX(G).$
Suppose $\rho$ is realized by $\Phi\in Formulas(G).$
Using Lemma~\ref{VycherkivanieSvobodnih} we obtain
a formula $\Phi_{1}\in Formulas(G)$
such that $\Phi_{1}$ realizes $\rho$ and
each unbound variable in $\Phi_{1}$ is used just once.
By Lemma \ref{svyasnyaformula}, there exists a connected subformula
$\Phi_{2}$ that realizes $\rho.$
Hence $\Phi_{2}\in GF(G)$ and $GF(G)$ is not empty.

\end{proof}

if $z_{0}= z_{m}$ then we say that a path $z_{0}\Psi_{1}z_{1}\Psi_{2}\ldots \Psi_{m}z_{m}$
is a \emph{cycle.}

We say that occurrences $\Psi_{1}$ and $\Psi_{2}$
of a formula $\Phi\in Formulas(G)$
\emph{are equivalent} if there exists a cycle that contains
$\Psi_{1}$ and $\Psi_{2}.$
In this case we write $\Psi_{1}\overset{\Phi}\sim\Psi_{2}.$

\begin{lemma}

Suppose $\Psi_{1}\overset{\Phi}\sim \Psi_{2},$
$\Psi_{2}\overset{\Phi} \sim \Psi_{3},$
then $\Psi_{1}\overset{\Phi}\sim \Psi_{3}.$

\end{lemma}

\begin{proof}

Suppose $C_{1}$ is a cycle that contains $\Psi_{1}$ and $\Psi_{2}.$
Let $V_{1}$ be the set of all variables in the cycle $C_{1}.$
Suppose $x_{1}\Psi_{2}y_{1}$ is a part of $C_{1}.$
We have a cycle $$C_{2} = z_{m}\Psi_{3}z_{1}\Xi_{1}z_{2}\Xi_{2}\ldots z_{m-1}\Xi_{m-1}z_{m}$$
that contains $\Psi_{2}$ and $\Psi_{3}.$
Hence, $x_{2}\Psi_{2}y_{2}$ is a part of $C_{2}$ for some variables $x_{2}$ and $y_{2}.$
Let $V_{2}$ be the set of all variables in the cycle $C_{2}.$
If $x_{1}\notin V_{2},$
then we replace
$x_{2}\Psi_{2}y_{2}$
by $x_{2}\Psi_{2}x_{1}\Psi_{2}y_{2}.$
Similarly, if $y_{1}\notin V_{2},$
then we replace
$x_{2}\Psi_{2}y_{2}$
by $x_{2}\Psi_{2}y_{1}\Psi_{2}y_{2}.$
Hence, without loss of generality it can be assumed that
$C_{2}$ contains $x_{1}$ and $y_{1}$ (see Figure \ref{threecycles}).

Let $p$ be the minimal number such that $z_{p}\in V_{1}.$
Let $q$ be the maximal number such that $z_{q}\in V_{1}.$
Since $C_{2}$ contains at least two variables $x_{1}$ and $y_{1}$ from $V_{1},$
we have $p<q.$
Since $z_{p},z_{q}\in V_{1},$
there exists a path
$$z_{p}\Delta_{1}t_{1}\Delta_{2}t_{2}\ldots \Delta_{s-1}t_{s-1}\Delta_{s}z_{q}$$
such that $t_{i}\in V_{1}$ for every $i,$
$\Delta_{i} = \Psi_{1}$ for some $i.$
Hence we have a cycle
$$
z_{m}\Psi_{3}z_{1}\Xi_{1}z_{2}\Xi_{2}\ldots \Xi_{p-1}z_{p}
\Delta_{1}t_{1}\Delta_{2}t_{2}\ldots \Delta_{s-1}t_{s-1}\Delta_{s}z_{q}
\Xi_{q}z_{q+1}\ldots \Xi_{m-1}z_{m}
$$
that contains $\Psi_{1}$ and $\Psi_{3}.$

\end{proof}

\begin{figure}
\centerline{
\xymatrix{
& z_{m-1} \ar@{-}[r]^{\Xi_{m-2}}& z_{m-2} \ar@{--}[r]& z_{q} \ar@{-}[r]^{\Delta_{s}}&t_{s-1}\ar@{--}[rd]&  \\
z_{m}\ar@{-}[ru] ^{\Xi_{m-1}} & \ar@{}[d]|{\txt{\large{cycle $C_{2}$}}}& x_{1}\ar@{--}[ru] & & \ar@{}[d]|{\txt{\large{cycle $C_{1}$}}}& t_{i} \ar@{-}[d]^{\Psi_{1}}\\
z_{1}\ar@{-}[u] ^{\Psi_{3}} &  & y_{1}\ar@{-}[u]^{\Psi_{2}} \ar@{--}[rd]& & & t_{i-1}\\
&z_{2}\ar@{-}[lu] ^{\Xi_{1}} & z_{3}\ar@{-}[l] ^{\Xi_{2}} & z_{p}\ar@{--}[l] \ar@{-}[r]^{\Delta_{1}}& t_{1}\ar@{--}[ru] &
}
}
\caption{}\label{threecycles}
\end{figure}

So, all occurrences of $\Phi\in GF(G)$ can be divided into
equivalence classes $U_{1},U_{2},\ldots,U_{m}.$
Such classes are called \emph{unions}.
By $Unions(\Phi)$ we denote the set of all unions in $\Phi.$
In the sequel we sometimes suppose that
a union $U$ is a subformula of $\Phi$ that contains all occurrences from $U$
and does not contain existential quantifiers.

We say that $\Phi$ is a \emph{tree-formula} if
every two occurrences of $\Phi$ are not equivalent.
This means that
every cycle in $\Phi$ contains just one occurrence.

\begin{lemma}\label{edinstvenostPeremennoyVUnion}

Suppose $x$ is a variable of $\Phi\in GF(G),$
$U$ is a union of $\Phi.$
Then there exists a unique variable $y\in Var(U)$
such that every path from $x$ to any variable $z\in Var(U)$ contains~$y.$

\end{lemma}

\begin{proof}

Suppose $z_{0}, t_{0}\in Var(U).$
Since $\Phi$ is connected,
there exist paths
$$z_{0}\Psi_{1}z_{1}\Psi_{2}z_{2}\Psi_{3}\ldots \Psi_{m-1}z_{m-1}\Psi_{m}z_{m},$$
$$t_{0}\Xi_{1}t_{1}\Xi_{2}t_{2}\Xi_{3}\ldots \Xi_{l-1}t_{l-1}\Xi_{l}t_{l},$$
where $z_{m} = t_{l} = x.$
Let $p$ be the maximal number such that
$z_{p}\in Var(U).$
Let $q$ be the maximal number such that
$t_{q}\in Var(U).$
To complete the proof we need to show that $z_{p} = t_{q}.$
In this case obviously $y = z_{p} = t_{q}$ and $y$ is unique.
Let us prove that $z_{p} = t_{q}.$

Assume the converse.
If $x\in Var(U)$, then $z_{p} = t_{q} = x$ and there is nothing to prove.
Suppose $x\notin Var(U),$ then $p<m,$ $q<l.$
Let $r$ be the minimal number
such that
$r>q$ and $t_{r} = z_{j}$ for some $j>p.$
Since $t_{l} = z_{m} = x,$ this number exists.
Since $U$ is a union,
there exists a path
$z_{p}\Delta_{1}s_{1}\Delta_{2}s_{2}\ldots s_{r-1}\Delta_{r} t_{q}$
such that $\Delta_{i}$ is an occurrence from $U$ for every $i.$
So, we have the cycle
\begin{multline*}z_{p}\Delta_{1}s_{1}\Delta_{2}s_{2}\ldots s_{r-1}\Delta_{r} t_{q}
\Xi_{q+1}t_{q+1}\Xi_{q+2}\ldots \\ \ldots \Xi_{r-1}t_{r-1}\Xi_{r}t_{r}
\Psi_{j}z_{j-1}\Psi_{j-1}z_{j-2}\ldots z_{p+1}\Psi_{p+1}z_{p}.\end{multline*}
Hence $\Delta_{1}\overset{\Phi}\sim \Xi_{q+1}.$ Then $t_{q+1} \in Var(U).$
This contradicts the maximality of~$q.$
The lemma is proved.

\end{proof}

Suppose $\Phi\in GF(G), U \in Unions(\Phi), x \in Var(\Phi).$
Let $KeyVar_{\Phi}(U,x)$ be a unique variable $y\in Var(U)$
such that every path from $x$ to any variable $z\in Var(U)$ contains $y.$
It follows from Lemma \ref{edinstvenostPeremennoyVUnion} that
$KeyVar_{\Phi}$ is well-defined.

Suppose $\Psi = \rho(z_{1},\ldots,z_{n})$ is an occurrence,
then we say that $\Psi$ has arity~$n.$
We denote the arity of $\Psi$ by $ar(\Psi).$

In the sequel we define difficult notions and parameters for formulas.
In this section we operate with formulas like with graphs.
That is why it seems
appropriate to explain new notions and parameters on the graph.
We give an example of a formula $\widetilde \Phi$ in Figure \ref{bigExample}.
We suppose that every occurrence in the formula $\widetilde \Phi$ has arity 2.
Variables of $\widetilde \Phi$ are vertexes in the figure, occurrences of $\widetilde \Phi$ are edges in the figure.
To distinguish bound variables we locate them inside circles.

\begin{figure}
\centerline{
\xymatrix{
*++[o][F-]{x_{1}} \ar@{-}[d]_{\Xi_{1}} & & *++[o][F-]{x_{2}} \ar@{-}[dl]_{\Xi_{2}}&  & *++[o][F-]{x_{3}}
\ar@{-}[dr]^{\Xi_{3}}&  & *++[o][F-]{x_{4}} \ar@{-}[dl]_{\Xi_{4}} &  \\
{z_{1}} \ar@{-}[r]^{\Xi_{5}} \ar@{-}[d]_{\Xi_{6}}
&  {z_{2}} \ar@{-}[ld]_{\Xi_{7}} \ar@{-}[rd]^{\Xi_{8}} &  & {z_{3}}
\ar@{-}[dl]_{\Xi_{9}} \ar@{-}[dr]^{\Xi_{10}}&  & {z_{4}} \ar@{-}[dl]_{\Xi_{11}} \ar@{-}[dr]^{\Xi_{12}} &  &  \\
z_{5} \ar@{-}[rr]^{\Xi_{13}} \ar@{-}[d]_{\Xi_{14}} & & z_{6} \ar@{-}[d]^{\Xi_{15}} \ar@{-}[rr]^{\Xi_{16}} &
 & z_{7} \ar@{-}[rd]^{\Xi_{17}}& & z_{8}\ar@{-}[ld]_{\Xi_{18}} \ar@{-}[r]^{\Xi_{19}}& *++[o][F-]{x_{5}} \\
*++[o][F-]{x_{6}} & & *++[o][F-]{x_{7}} & & *++[o][F-]{x_{8}} & z_{9} \ar@{-}[l]_{\Xi_{20}} \ar@{-}[r]^{\Xi_{21}} & *++[o][F-]{x_{9}} &
}
}
\caption{}\label{bigExample}
\end{figure}

Formula $\widetilde \Phi$ in Figure \ref{bigExample} has 9 bound variables and 9 unbound variables.
There are 12 unions in $\widetilde \Phi$:
$U_{1} = \{\Xi_{5},\Xi_{6},\Xi_{7},\Xi_{8},\Xi_{13}\},$
$U_{2} = \{\Xi_{9},\Xi_{10},\Xi_{16}\},$
$U_{3} = \{\Xi_{11},\Xi_{12},\Xi_{17},\Xi_{18}\},$
$U_{4} = \{\Xi_{1}\},$
$U_{5} = \{\Xi_{2}\},$
$U_{6} = \{\Xi_{3}\},$
$U_{7} = \{\Xi_{4}\},$
$U_{8} = \{\Xi_{19}\},$
$U_{9} = \{\Xi_{14}\},$
$U_{10} = \{\Xi_{15}\},$
$U_{11} = \{\Xi_{20}\},$
$U_{12} = \{\Xi_{21}\}.$
Some values for $KeyVar$ are listed below:
$KeyVar_{\widetilde \Phi}(U_{1},x_{1})=z_{1},$
$KeyVar_{\widetilde \Phi}(U_{1},x_{5})=z_{6},$
$KeyVar_{\widetilde \Phi}(U_{3},x_{7})=z_{7}.$

Suppose $U$ is a union of $\Phi,$
$X = \{x_{1},\ldots,x_{n}\} = UVar(\Phi).$
Let $y\in Var(U),$ then
we put
$$KeySet_{\Phi}(U,y) = \{x \in X\;|\; KeyVar_{\Phi}(U,x) = y\},$$
$$MaxKeySet_{\Phi}(U) = \max \limits _{y\in Var(U)} |KeySet_{\Phi}(U,y)|.$$
For example (see Figure \ref{bigExample}),
$$KeySet_{\widetilde \Phi}(U_{1},z_{1}) = \{x_{1}\}, \;\;
KeySet_{\widetilde \Phi}(U_{2},z_{6}) = \{x_{1},x_{2},x_{6},x_{7}\},$$
$$KeySet_{\widetilde \Phi}(U_{2},z_{3}) = \varnothing, \;\;
KeySet_{\widetilde \Phi}(U_{1},z_{6}) = \{x_{3},x_{4},x_{5},x_{7},x_{8},x_{9}\}.$$
$$MaxKeySet_{\widetilde \Phi}(U_{1}) = 6, \;\;
MaxKeySet_{\widetilde \Phi}(U_{2}) = 5,$$
$$MaxKeySet_{\widetilde \Phi}(U_{3}) = 4, \;\;
MaxKeySet_{\widetilde \Phi}(U_{4}) = 8.$$

Put
$$Arity(U) = \sum\limits_{\Psi\in U} (ar(\Psi)-1),$$
$$Par(U) = 2+ Arity(U) - |Var(U)|,$$
For example (see Figure \ref{bigExample}),
$$Arity(U_{1}) = 5, \;\; Par(U_{1}) = 2+5 -4 = 3,$$
$$Arity(U_{2}) = 3, \;\; Par(U_{2}) = 2+3 -3 = 2,$$
$$Arity(U_{3}) = 4, \;\; Par(U_{3}) = 2+4 -4 = 2,$$
$$Arity(U_{4}) = 1, \;\; Par(U_{4}) = 2+1 -2 = 1.$$

We say that $U$ is a \emph{trivial union} if there exist $z_{1},z_{2}\in Var(U)$   
such that $KeySet_{\Phi}(U,z_{1})\cup KeySet_{\Phi}(U,z_{2}) = UVar(\Phi).$
All other unions are called \emph{nontrivial} and the set of all nontrivial unions in $\Phi$ is denoted by
$NUnions(\Phi).$
$\widetilde \Phi$ has just two nontrivial unions: $U_{1}$ and $U_{3}$ (see Figure \ref{bigExample}).
For example union $U_{2}$ is trivial, because
$$KeySet_{\widetilde \Phi}(U_{2},z_{6})\cup
KeySet_{\widetilde \Phi}(U_{2},z_{7})  =
\{x_{1},x_{2},x_{6},x_{7}\}\cup
\{x_{3},x_{4},x_{5},x_{8},x_{9}\}.$$ 

$$MaxUnion(\Phi) = \max\limits_{U\in NUnions(\Phi)} |U|,$$
$$MaxPar(\Phi) = \max\limits_{\substack{U\in NUnions(\Phi)\\ |U| = MaxUnion(\Phi)}} Par(U).$$
Also we put $MaxUnion(\Phi)=1$ if $NUnions(\Phi)=\varnothing.$
By $Max(\Phi)$ we denote the set of all
unions $U\in NUnions(\Phi)$
such that $|U|=MaxUnion(\Phi)$ and
$Par(U)=MaxPar(\Phi).$ Put
$$MinMaxKeySet(\Phi) = \min\limits_{U\in Max(\Phi)} MaxKeySet_{\Phi}(U).$$
By $Number(\Phi)$ we denote the number of unions $U\in Max(\Phi)$
such that 
$MaxKeySet_{\Phi}(U) = MinMaxKeySet(\Phi).$

Obviously $MaxUnion(\widetilde \Phi) = 5,$ $MaxPar(\widetilde \Phi) = 3$ (see Figure \ref{bigExample}).
Also $Max(\widetilde \Phi)=\{U_{1}\},$
$MinMaxKeySet(\widetilde \Phi) = 6,$
$Number(\widetilde \Phi) = 1.$

So, we have the mapping $\Upsilon: GF(G) \longrightarrow \mathbb N \times \mathbb N \times \mathbb N \times \mathbb N$
such that
$$\Upsilon(\Phi) = (MaxUnion(\Phi),MaxPar(\Phi),MinMaxKeySet(\Phi),Number(\Phi)).$$
If $NUnions(\Phi) = \varnothing,$ then we suppose that $\Upsilon(\Phi)$ is not defined.

Let us define a linear order on the set $\mathbb N \times \mathbb N \times \mathbb N \times \mathbb N.$
We say that $(a_{1},a_{2},a_{3},a_{4})\le (b_{1},b_{2},b_{3},b_{4})$ iff
\begin{multline*}(a_{1}<b_{1}) \vee (a_{1} = b_{1} \wedge a_{2}<b_{2}) \vee  (a_{1} = b_{1} \wedge a_{2}=b_{2}
\wedge a_{3}> b_{3}) \vee \\
\vee  (a_{1} = b_{1} \wedge a_{2}=b_{2}   \wedge a_{3}=b_{3}
\wedge a_{4}\le b_{4}).\end{multline*}

\begin{lemma}\label{posledovatelnostKonechna}

Suppose
$(a_{i+1},b_{i+1},c_{i+1},d_{i+1})\le
(a_{i},b_{i},c_{i},d_{i}),$
$c_{i}<C$
for every $i\in \mathbb N.$
Then there exists $n\in \mathbb N,$
such that
$(a_{i},b_{i},c_{i},d_{i}) = (a_{i+1},b_{i+1},c_{i+1},d_{i+1})$
for every $i>n.$

\end{lemma}

\begin{proof}

Suppose $a$ is the minimal number such that
$a=a_{p}$ for some $p.$
Hence $a_{i} = a$ for every $i\ge p.$
Suppose $b$ is the minimal number such that
$b=b_{q}$ for some $q\ge p.$
Hence $b_{i} = b$ for every $i\ge q.$
Suppose $c$ is the maximal number such that
$c_{r} = c$ for some $r\ge q.$
Hence $c_{i} = c$ for every $i\ge r.$
Suppose $d$ is the minimal number such that
$d_{s} = d$ for some $s\ge r.$
Hence $a_{i} = a,$
$b_{i} = b,$
$c_{i} = c,$
$d_{i} = d,$
for every $i\ge s.$
This concludes the proof.

\end{proof}


By $RQ(\Phi,y_{1},y_{2},\ldots,y_{l})$ we denote the formula that is obtained from $\Phi$
by removing the expressions $\exists y_{1},\exists y_{2},\ldots,\exists y_{l}.$ 

By $RVS(x_{1},\ldots,x_{n},\Psi, x_{1}',\ldots,x_{n}')$
we denote
the formula that is obtained from $\Phi$
by renaming variable $x_{i}$ by the variable $x_{i}'$
for every $i.$

Suppose $y$ is an unbound variable of $\Phi\in GF(G),$ $\Psi$ is an occurrence of~$\Phi.$
By $RV(\Phi, y, z, t,  \Psi)$ we denote the formula that is obtained from $\Phi$
by renaming variable $y$ in the occurrence $\Psi$ 
by variable $z,$
and renaming variable $y$ in other occurrences
by variable $t.$

\textbf{Transformation.}
Suppose a formula $\Phi \in GF(G),$ $\Phi(x_{1},\ldots,x_{n})$ realizes a predicate $\rho\in MAX(G).$
Suppose a variable $y$ is a bound variable in $\Phi$ and this variable is
used in occurrences $\Psi.$ 
Let
$$\rho_{0}(y) = \exists x_{1}\ldots \exists x_{n} RQ(\Phi, y).$$
By $\Phi_{i}'$  we denote the following formula:
$$RVS(x_{1},x_{2},\ldots,x_{n},RV(RQ(\Phi, y), y, y_{i}, y_{i+1}, \Psi),x_{i,1},x_{i,2},\ldots,x_{i,n}).$$
By $\Phi'$  we denote the following formula:
$$\exists y_{0} \exists y_{1}\ldots \exists y_{k} \;\left( \bigwedge \limits_{0\le i \le k-1}
\Phi_{i}'\right)\wedge \left( \bigwedge \limits_{0\le i \le k}\rho_{0}(y_{i})\right).$$

Let
\begin{multline*}\rho'(x_{0,1},\ldots,x_{0,n}, x_{1,1},\ldots,x_{1,n},
\ldots,x_{k-1,1},\ldots,x_{k-1,n}) = \\ =\Phi'(x_{0,1},\ldots,x_{0,n}, x_{1,1},\ldots,x_{1,n},
\ldots,x_{k-1,1},\ldots,x_{k-1,n}).\end{multline*}
By Theorem \ref{GlavnoePreobrazovanie},
we have
$\rho'(x_{1},\ldots,x_{n},
\ldots,x_{1},\ldots,x_{n}) =
\rho(x_{1},\ldots,x_{n}).$
Hence by Lemma \ref{VycherkivanieSvobodnih}, it follows that
$\rho \in Strike(\rho')$ and
there exist $i_{1},\ldots,i_{n} \in \{0,1,\ldots,k-1\}$ such that
$$\rho(x_{1},\ldots,x_{n}) =
\exists t_{1} \exists t_{2}\ldots \exists  t_{l} \;  \rho'(z_{0,1},\ldots,z_{0,n},
\ldots,z_{k-1,1},\ldots,z_{k-1,n}),$$
where $z_{i_{j},j} = x_{j}$ for every $j,$
$z_{i,j}$ are different for every $i,j.$
By $D_{m}$ we denote the set of all $j\in \{1,\ldots,n\}$
such that $i_{j} = m.$
Hence,
$D_{0}\sqcup D_{1}\sqcup \ldots \sqcup D_{k-1} = \{1,2,\ldots,n\}.$

Suppose $m\in \{0,1,\ldots,k-1\},$
$\{1,2,\ldots,n\} \setminus D_{m} = \{l_{1},l_{2},\ldots,l_{s}\}.$
Then by $\Phi_{m}$ we denote the following formula:
$$\exists x_{l_{1}} \exists x_{l_{2}} \ldots \exists x_{l_{s}}
 \;  RV(RQ(\Phi, y), y, y_{m}, y_{m+1}, \Psi).$$

By $Trans(\Phi, y, \Psi)$ we denote the
formula
$$\exists y_{0} \exists y_{1}\ldots \exists y_{k} \; \left(\bigwedge \limits_{0\le i \le k-1}
\Phi_{i} \right)
\wedge \left( \bigwedge \limits_{0\le i \le k}\rho_{0}(y_{i})\right).$$
It can be proved that
every unbound variable of
$Trans(\Phi, y, \Psi)$ is used just once and
$Trans(\Phi, y, \Psi)$ realizes predicate $\rho.$
Hence $Trans(\Phi, y, \Psi)\in GF(G).$

In Figure~\ref{TransExample} you can find
an example of applying the transformation.
One of the possible results of calculating
$\Omega = Trans(\widetilde \Phi,z_{5},\Xi_{6})$ is presented on this figure.
Here we suppose that $k=3$ and $\widetilde \Phi$ is the formula from Figure~\ref{bigExample}.
Roughly speaking, $\Omega$ consists of three copies of $\widetilde \Phi$
and four new occurrences $\Theta_{i} = \rho_{0}(y_{i})$ for $i=0,1,2,3.$

There are only three nontrivial unions in $\Omega$:
$W_{0} = \{\Xi_{0,7},\Xi_{0,8},\Xi_{0,13}\},$
$W_{1} = \{\Xi_{1,7},\Xi_{1,8},\Xi_{1,13}\},$
$W_{2} = \{\Xi_{2,7},\Xi_{2,8},\Xi_{2,13}\}.$
It can be easily checked that
$$Par(W_{0}) = Par(W_{1}) = Par(W_{2}) = 2 + 3 - 3 = 2,$$
$$MaxKeySet_{\Omega}(W_{0}) = 7,MaxKeySet_{\Omega}(W_{1}) = 4,MaxKeySet_{\Omega}(W_{2}) = 7.$$
Hence $$MaxUnion(\Omega) = 3,\;\;
MaxPar(\Omega) = 2,$$
$$MinMaxKeySet(\Omega) = 4,\;\;
Number(\Omega) = 1.$$
So, we have $\Upsilon(\widetilde \Phi) = (5,3,6,1),$
$\Upsilon(\Omega) = (3,2,4,1)$ and $\Upsilon(\Omega)<\Upsilon(\widetilde \Phi).$
Using the transformation we reduce value of $\Upsilon.$

\begin{figure}
\centerline{
\xymatrix{
&{x_{0,1}} \ar@{-}[d]^{\Xi_{0,1}} & & {x_{0,2}} \ar@{-}[dl]_{\Xi_{0,2}}&  & *++[o][F-]{x_{3}}
\ar@{-}[dr]^{\Xi_{0,3}}&  & *++[o][F-]{x_{4}} \ar@{-}[dl]^{\Xi_{0,4}} \\
y_{0} \ar@{-}@(ul,ur)[]^(0.55){\Theta_{0}} &{z_{0,1}} \ar@{-}[r]^{\Xi_{0,5}} \ar@{-}[l]^{\Xi_{0,6}}
&  {z_{0,2}} \ar@{-}[ld]_{\Xi_{0,7}} \ar@{-}[rd]^{\Xi_{0,8}} &  & {z_{0,3}}
\ar@{-}[dl]_{\Xi_{0,9}} \ar@{-}[dr]^{\Xi_{0,10}}&  & {z_{0,4}} \ar@{-}[dl]_{\Xi_{0,11}} \ar@{-}[dr]^{\Xi_{0,12}} & {x_{0,5}} \\
&y_{1}  \ar@{-}@(l,u)[]^(0.35){\Theta_{1}} \ar@{-}[rr]^{\Xi_{0,13}} \ar@{-}[d]^{\Xi_{0,14}} & & z_{0,6} \ar@{-}[d]^{\Xi_{0,15}} \ar@{-}[rr]^{\Xi_{0,16}} &
 & z_{0,7} \ar@{-}[rd]^{\Xi_{0,17}}& & z_{0,8}\ar@{-}[ld]_{\Xi_{0,18}} \ar@{-}[u]_{\Xi_{0,19}} \\
&*++[o][F-]{x_{6}} & & {x_{0,7}} & & {x_{0,8}} & z_{0,9} \ar@{-}[l]_{\Xi_{0,20}} \ar@{-}[r]^{\Xi_{0,21}} & {x_{0,9}} & \\
&{x_{1,1}} \ar@{-}[d]^{\Xi_{1,1}} & & *++[o][F-]{x_{2}} \ar@{-}[dl]_{\Xi_{1,2}}&  & *{x_{1,3}}
\ar@{-}[dr]^{\Xi_{1,3}}&  & *{x_{1,4}} \ar@{-}[dl]^{\Xi_{1,4}} \\
&{z_{1,1}} \ar@/^2pc/@{-}[uuu]^{\Xi_{1,6}} \ar@{-}[r]^{\Xi_{1,5}}
&  {z_{1,2}} \ar@{-}[ld]_{\Xi_{1,7}} \ar@{-}[rd]^{\Xi_{1,8}} &  & {z_{1,3}}
\ar@{-}[dl]_{\Xi_{1,9}} \ar@{-}[dr]^{\Xi_{1,10}}&  & {z_{1,4}} \ar@{-}[dl]_{\Xi_{1,11}} \ar@{-}[dr]^{\Xi_{1,12}} & *++[o][F-]{x_{5}} \\
&y_{2}  \ar@{-}@(l,u)[]^(0.35){\Theta_{2}} \ar@{-}[rr]^{\Xi_{1,13}} \ar@{-}[d]^{\Xi_{1,14}} & & z_{1,6} \ar@{-}[d]^{\Xi_{1,15}} \ar@{-}[rr]^{\Xi_{1,16}} &
 & z_{1,7} \ar@{-}[rd]^{\Xi_{1,17}}& & z_{1,8}\ar@{-}[ld]_{\Xi_{1,18}} \ar@{-}[u]_{\Xi_{1,19}} \\
&{x_{1,6}} & & {x_{1,7}} & & {x_{1,8}} & z_{1,9} \ar@{-}[l]_{\Xi_{1,20}} \ar@{-}[r]^{\Xi_{1,21}} & {x_{1,9}} & \\
&*++[o][F-]{x_{1}} \ar@{-}[d]^{\Xi_{2,1}} & & {x_{2,2}} \ar@{-}[dl]_{\Xi_{2,2}}&  & {x_{2,3}}
\ar@{-}[dr]^{\Xi_{2,3}}&  & {x_{2,4}} \ar@{-}[dl]^{\Xi_{2,4}} \\
&{z_{2,1}} \ar@/^2pc/@{-}[uuu]^{\Xi_{2,6}} \ar@{-}[r]^{\Xi_{2,5}}
&  {z_{2,2}} \ar@{-}[ld]_{\Xi_{2,7}} \ar@{-}[rd]^{\Xi_{2,8}} &  & {z_{2,3}}
\ar@{-}[dl]_{\Xi_{2,9}} \ar@{-}[dr]^{\Xi_{2,10}}&  & {z_{2,4}} \ar@{-}[dl]_{\Xi_{2,11}} \ar@{-}[dr]^{\Xi_{2,12}} & {x_{2,5}} \\
&y_{3} \ar@{-}@(l,u)[]^(0.35){\Theta_{3}} \ar@{-}[rr]^{\Xi_{2,13}} \ar@{-}[d]^{\Xi_{2,14}} & & z_{2,6} \ar@{-}[d]^{\Xi_{2,15}} \ar@{-}[rr]^{\Xi_{2,16}} &
 & z_{2,7} \ar@{-}[rd]^{\Xi_{2,17}}& & z_{2,8}\ar@{-}[ld]_{\Xi_{2,18}} \ar@{-}[u]_{\Xi_{2,19}} \\
&{x_{2,6}} & & *++[o][F-]{x_{7}} & & *++[o][F-]  {x_{8}} & z_{2,9} \ar@{-}[l]_{\Xi_{2,20}} \ar@{-}[r]^{\Xi_{2,21}} & *++[o][F-]{x_{9}} &
}
}
\caption{Example of applying the transformation: $Trans(\widetilde \Phi,z_{5},\Xi_{6})$}\label{TransExample}
\end{figure}
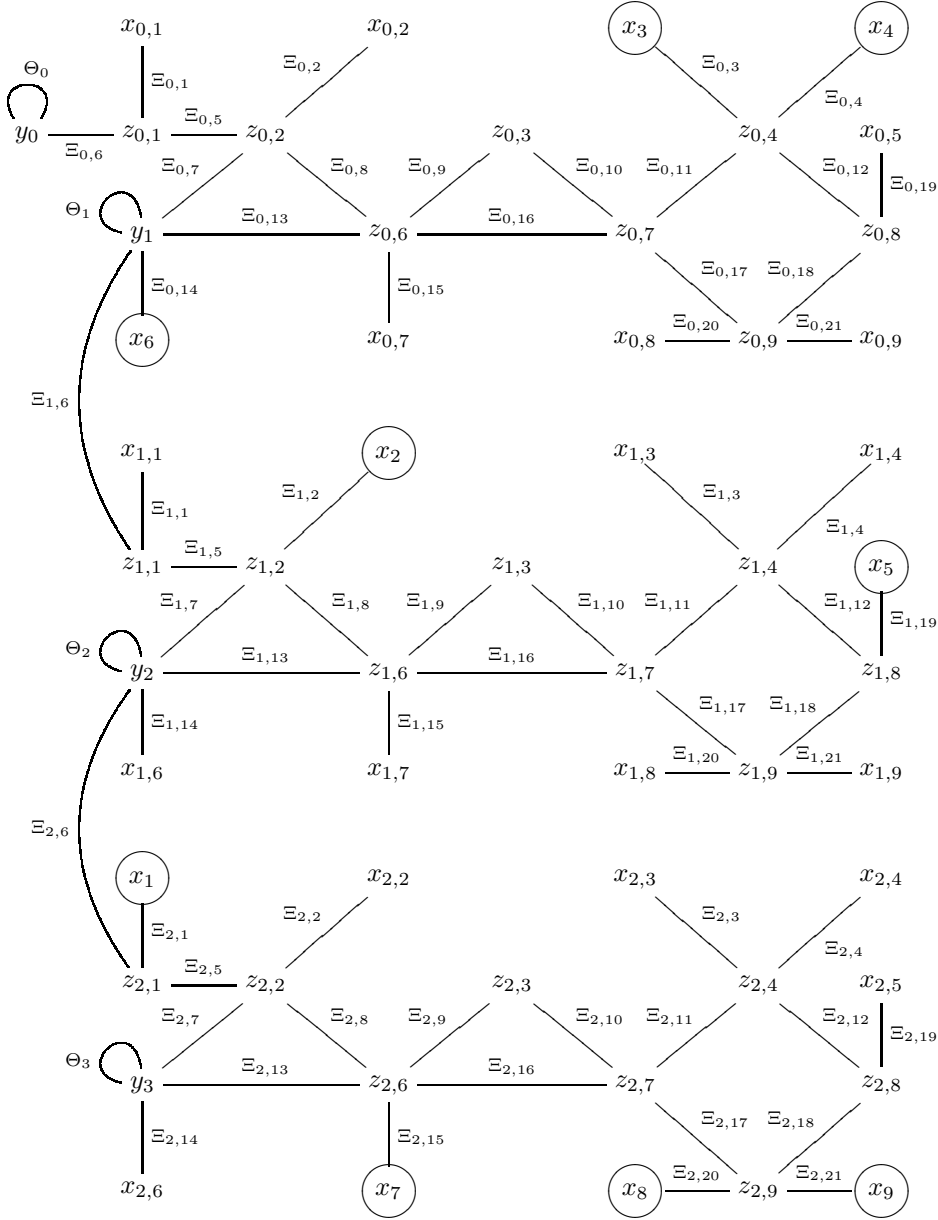

In the following lemmas
we will use notations from the definition of the transformation.
As it follows from the definition of $Trans(\Phi, y, \Psi)$
we create $k$ copies of every occurrence from $\Phi.$
We suppose that all bound variables in formulas
$\Phi_{0},\ldots,\Phi_{k-1}$ are different.
Nevertheless, to simplify explanations
we use the same notations for bound variables in all copies.





We say that a union $U$ in $\Phi$ is a \emph{parent} of a union $U'$ in $Trans(\Phi, y, \Psi)$
iff
there exists $m$ such that the following conditions hold:

1) $U'$ is a union of $\Phi_{m};$

2) $U$ is obtained from $U'$
by replacing variables $y_{m},y_{m+1}$ by variable $y.$

\begin{lemma}\label{unionlejitvkusochke}

Suppose $U$ is a union of $Trans(\Phi, y, \Psi).$
Then there exists $m\in \{0,1,\ldots,k-1\}$ such that $U$ is a union in
$\Phi_{m}.$

\end{lemma}

\begin{proof}

Assume the converse.
Suppose we have two occurrences $\Psi_{1}$ and $\Psi_{2}$
such that $\Psi_{1}\in \Phi_{i},$
$\Psi_{2}\in \Phi_{j},$ $i< j.$
Obviously, every path that contains $\Psi_{1}$ and $\Psi_{2}$
also contains variable $y_{i+1}.$
Hence there is no a cycle that contains $\Psi_{1}$ and $\Psi_{2}.$
Therefore, $\Psi_{1}$ and $\Psi_{2}$ are not equivalent.
This concludes the proof.

\end{proof}

The following lemma follows from the definition of $Trans(\Phi, y, \Psi).$

\begin{lemma}\label{ekvivalentnostostaetsya}

Suppose $\Psi_{1}$ and $\Psi_{2}$ are occurrences in $\Phi_{m},$
$\Psi_{1}\overset{\Phi_{m}}\sim\Psi_{2},$
$\Psi_{1}'$ and $\Psi_{2}'$ are obtained from $\Psi_{1}$ and $\Psi_{2}$
by replacing variables $y_{m},y_{m+1}$ by variable~$y.$
Then
$\Psi_{1}'$ and $\Psi_{2}'$ are occurrences in $\Phi$ and
$\Psi_{1}'\overset{\Phi}\sim\Psi_{2}'.$


\end{lemma}

The following lemma states that a trivial union is transformed into a trivial union.

\begin{lemma}\label{trivialnieOstayutsya}

Suppose $U$ is a trivial union in $Trans(\Phi, y, \Psi),$
$\Psi\in Q$ and $Q \in NUnions(\Phi),$
$m\in \{0,1,\ldots,k-1\},$
$U'$ is obtained from $U$
by replacing variable $y$ by variable $y_{m+1}.$
Then $U'$ is a parent of $U,$ and $U'$ is a trivial union in
$Trans(\Phi, y, \Psi),$

\end{lemma}

\begin{proof}

Since $U$ is trivial,
there exists $z_{1},z_{2}\in Var(U)$ such that
$$KeySet_{\Phi}(U,z_{1})\cup KeySet_{\Phi}(U,z_{2}) = UVar(\Phi).$$
Since $\Psi\in Q$ and $Q$ is nontrivial, we have
$KeyVar_{\Phi}(U,y)\in \{z_{1},z_{2}\}.$
Hence, it can be easily checked that $U'$ is a trivial union in
$Trans(\Phi, y, \Psi).$

\end{proof}

\begin{lemma}\label{estPredok}

Suppose $U$ is a nontrivial union of $Trans(\Phi, y, \Psi),$
$\Psi\in Q$ and $Q \in NUnions(\Phi),$
then one of the following conditions holds
\begin{enumerate}
\item  $|U|< MaxUnion(\Phi);$
\item $|U|= MaxUnion(\Phi),$
there exists $U_{0} \in NUnions(\Phi)$ such that $U_{0}$ is a parent of $U.$
\end{enumerate}

\end{lemma}

\begin{proof}

Suppose $U = \{\Psi_{1},\Psi_{2},\ldots,\Psi_{r}\}.$
It follows from Lemma \ref{unionlejitvkusochke} that
$U$ is a union of $\Phi_{m}$
for some $m\in\{0,1,\ldots,k-1\}.$
Suppose $\Psi_{i}'$ is obtained from $\Psi_{i}$
by replacing variables $y_{m},y_{m+1}$ by variable $y.$
It follows from Lemma \ref{ekvivalentnostostaetsya} that
$\Psi_{i}'\overset{\Phi}\sim \Psi_{j}'$ for every $i,j.$
Hence there exists a union $U_{0}$ of $\Phi$
such that $\{\Psi_{1}',\Psi_{2}',\ldots,\Psi_{r}'\}\subseteq U_{0}.$
By Lemma \ref{trivialnieOstayutsya}, $U_{0}$ cannot be trivial.
Therefore 
$$|U| \le |U_{0}| \le MaxUnion(\Phi).$$

If $|U| < MaxUnion(\Phi),$ then the lemma is proved.
Suppose that $|U| = MaxUnion(\Phi).$
Then $\{\Psi_{1}',\Psi_{2}',\ldots,\Psi_{r}'\}= U_{0}$
and $U_{0}$ is a parent of $U.$
This completes the proof.

\end{proof}

%
%
%
%
%
%

\begin{lemma}\label{UnionNeSoderjitPsi}

Suppose $U\in Unions(\Phi_{m}),$
$U_{0}\in Unions(\Phi)$ is a parent of $U,$
$\Psi\notin U_{0}.$
Then $Par(U) = Par(U_{0}).$

\end{lemma}

\begin{proof}

Since $U_{0}$ is a parent, we have
$Arity(U) = Arity(U_{0}).$
Also, $|Var(U)|= |Var(U_{0})|$ because $\Psi\notin U_{0}.$
This completes the prove.

\end{proof}

\begin{lemma}\label{UnionSoderjitPsi}

Suppose $U\in Unions(\Phi_{m}),$
$U_{0}\in Unions(\Phi)$ is a parent of $U,$
$\Psi\notin U_{0},$
$y$ is used in $U_{0}$ at least twice.
Then $Par(U) = Par(U_{0})-1.$

\end{lemma}

\begin{proof}

Since $U_{0}$ is a parent, we have
$Arity(U) = Arity(U_{0}).$
Also, $|Var(U)|= |Var(U_{0})|+1$
because we have $y_{m}$ and $y_{m+1}$ in $U$ instead of $y$ in $U_{0}.$

\end{proof}

\begin{lemma}\label{minmaxkeyset}

Suppose $U_{0},U\in Max(\Phi),$ %
$U\neq U_{0},$
$y \in Var(U_{0}),$
$$MaxKeySet_{\Phi}(U_{0})=MinMaxKeySet(\Phi).$$
Then $$|KeySet_{\Phi}(U,KeyVar_{\Phi}(U,y))| = MaxKeySet_{\Phi}(U).$$

\end{lemma}

\begin{proof}

Let $\{z_{1},z_{2},\ldots,z_{r}\} = \{z\in Var(U) \;|\; KeySet_{\Phi}(U,z_{i})\neq \varnothing\}.$
Hence $UVar(\Phi) = \bigcup \limits_{i} KeySet_{\Phi}(U,z_{i}).$

\begin{figure}
\centerline{
\xymatrix{
*++[o][F-]{x_{1}} \ar@{-}[d]_{\Xi_{1}}& {z_{1}}\ar@{}[d]|(0.7){\txt{Union $U$}} \ar@{-}[ld]_{\Xi_{2}}\ar@{-}[rd]^{\Xi_{3}} & *++[o][F-]{x_{2}}\ar@{-}[l]_{\Xi_{4}}&
 *++[o][F-]{x_{3}} \ar@{-}[d]^{\Xi_{5}}& & y \ar@{}[d]|(0.7){\txt{Union $U_{0}$}}& *++[o][F-]{x_{4}} \ar@{-}[l]_{\Xi_{6}}& \\
z_{2} \ar@{-}[rr]_{\Xi_{7}}& & {z_{3}=t} \ar@{-}[r]^{\Xi_{8}}& y_{1} \ar@{-}[r]^{\Xi_{9}}&
t_{0} \ar@{-}[ru]^{\Xi_{11}} \ar@{-}[rr]_{\Xi_{10}}& & y_{2}  \ar@{-}[lu]_{\Xi_{12}} \ar@{-}[r]^{\Xi_{13}}& *++[o][F-]{x_{5}}
}
}
\caption{}\label{TwoUnions}
\end{figure}

Let $t = KeyVar_{\Phi}(U,y),$
$t_{0} = KeyVar_{\Phi}(U_{0},t)$
(see an example in Figure~\ref{TwoUnions}).
Perhaps $t_{0} = t.$
Assume that $$|KeySet_{\Phi}(U,t)| < MaxKeySet_{\Phi}(U),$$
then $KeySet_{\Phi}(U,z_{i})= MaxKeySet_{\Phi}(U)$ for some $z_{i}\neq t.$
Since $U$ is nontrivial, we have $r>2$ and
$\sum\limits_{z_{i}\neq t} |KeySet_{\Phi}(U,z_{i})|> MaxKeySet_{\Phi}(U).$

For every $z_{i}$ such that $z_{i}\neq t$
and for every $x\in KeySet_{\Phi}(U,z_{i})$
there exists a path from $t_{0}$ to $x$ that does not contain any variables from $U_{0}$ except $t_{0}.$
Hence $KeySet_{\Phi}(U,z_{i})\subseteq KeySet_{\Phi}(U_{0},t_{0})$ for every $z_{i}\neq t.$ Therefore,
\begin{multline*}MaxKeySet_{\Phi}(U_{0})\ge |KeySet_{\Phi}(U_{0},t_{0})|\ge \\ \ge
\sum\limits_{z_{i}\neq t} |KeySet_{\Phi}(U,z_{i})|> MaxKeySet_{\Phi}(U).\end{multline*}
This contradicts the condition $MaxKeySet_{\Phi}(U_{0})=MinMaxKeySet(\Phi).$
So, we have $|KeySet_{\Phi}(U,t)| = MaxKeySet_{\Phi}(U).$

\end{proof}

\begin{lemma}\label{unionMaxKeySet}

Suppose $U_{0},U\in Max(\Phi),$ $U\neq U_{0},$
$\Psi\in U_{0},$
$\Omega = Trans(\Phi, y, \Psi),$
$MaxKeySet_{\Phi}(U_{0})=MinMaxKeySet(\Phi).$
Let $W$ be the set of all $U'$ such that
$U'$ is nontrivial and $U$ is a parent of $U'.$

Then $MaxKeySet_{\Omega}(U')\ge MaxKeySet_{\Phi}(U)$ for every $U'\in W.$
Moreover, if $|W|>1$ then
$MaxKeySet_{\Omega}(U')>MaxKeySet_{\Phi}(U)$ for every $U'\in W.$

\end{lemma}

\begin{proof}

Suppose $W = \{U_{1},\ldots,U_{s}\}.$
Suppose $U_{i}$ is a union of $\Phi_{p_{i}}$ for every $i.$

Let $\{z_{1},\ldots,z_{r}\} = \{z\in Var(U) \;|\; KeySet_{\Phi}(U,z)\neq \varnothing\}.$
Hence $$UVar(\Phi) = \bigcup \limits_{i} KeySet_{\Phi}(U,z_{i}).$$

Let $t = KeyVar_{\Phi}(U,y),$
$t_{i} = KeyVar_{\Omega}(U_{i},y_{p_{i}+1}).$

By Lemma \ref{minmaxkeyset}, we have $|KeySet_{\Phi}(U,t)| = MaxKeySet_{\Phi}(U).$

It can be easily checked that
\begin{multline*}
MaxKeySet_{\Omega}(U_{i})\ge |KeySet_{\Omega}(U_{i},t_{i})|\ge \\ \ge |KeySet_{\Phi}(U,t)| =  MaxKeySet_{\Phi}(U) 
\end{multline*}
Since $U_{i}$  is nontrivial,
we have $|KeySet_{\Phi}(U_{i},t_{i})|<|UVar(\Omega)|.$
If $|W|>1$ then for $i,j\le |W|,$ $i\neq j$  we have
\begin{multline*}MaxKeySet_{\Omega}(U_{i})\ge |KeySet_{\Omega}(U_{i},t_{i})|\ge \\ 
\ge |KeySet_{\Phi}(U,t)| +(|UVar(\Omega)|-|KeySet_{\Phi}(U_{j},t_{j})|)>\\
> |KeySet_{\Phi}(U,t)| =  MaxKeySet_{\Phi}(U)
\end{multline*}
The lemma is proved.

\end{proof}

Suppose $\Phi'$ is a subformula of $\Phi,$
$y\in Var(\Phi').$
By $Part(\Phi,\Phi',y)$ we denote the formula defined by following conditions:

\begin{enumerate}

\item $Part(\Phi,\Phi',y)$ contains occurrence $\Psi$ iff
\begin{enumerate}

\item $\Phi$ contains occurrence $\Psi,$

\item $\Phi'$ does not contain occurrence $\Psi,$

\item 
there exists a path in $\Phi$ from $y$ to some variable of $\Psi$
that does not contain any occurrences from $\Phi';$

\end{enumerate}

\item $Part(\Phi,\Phi',y)$ contains $\exists z$ iff

\begin{enumerate}

\item
$\Phi$ contains $\exists z,$

\item
$z\notin Var(\Phi'),$

\item
there exists a path in $\Phi$ from $y$ to $z$ that does not contain any occurrences from $\Phi'.$

\end{enumerate}

\end{enumerate}

This notion looks complicated, but it is not. In Figure~\ref{partExample} you
can see $Part(\widetilde \Phi, U_{1}, z_{6}),$
where $\widetilde \Phi$ is the formula in Figure \ref{bigExample},
$U_{1} = \{\Xi_{5},\Xi_{6},\Xi_{7},\Xi_{8},\Xi_{13}\}.$

\begin{figure}
\centerline{
\xymatrix{
&  & *++[o][F-]{x_{3}}
\ar@{-}[dr]^{\Xi_{3}}&  & *++[o][F-]{x_{4}} \ar@{-}[dl]_{\Xi_{4}} &  \\
 & {z_{3}} \ar@{-}[dl]_{\Xi_{9}} \ar@{-}[dr]^{\Xi_{10}}&  & {z_{4}} \ar@{-}[dl]_{\Xi_{11}} \ar@{-}[dr]^{\Xi_{12}} &  &  \\
z_{6} \ar@{-}[d]^{\Xi_{15}} \ar@{-}[rr]^{\Xi_{16}} & & z_{7} \ar@{-}[rd]^{\Xi_{17}}& & z_{8}\ar@{-}[ld]_{\Xi_{18}} \ar@{-}[r]^{\Xi_{19}}& *++[o][F-]{x_{5}} \\
*++[o][F-]{x_{7}} & & *++[o][F-]{x_{8}} & z_{9} \ar@{-}[l]_{\Xi_{20}} \ar@{-}[r]^{\Xi_{21}} & *++[o][F-]{x_{9}} &
}
}
\caption{$Part(\widetilde \Phi, U_{1}, z_{6})$}\label{partExample}
\end{figure}
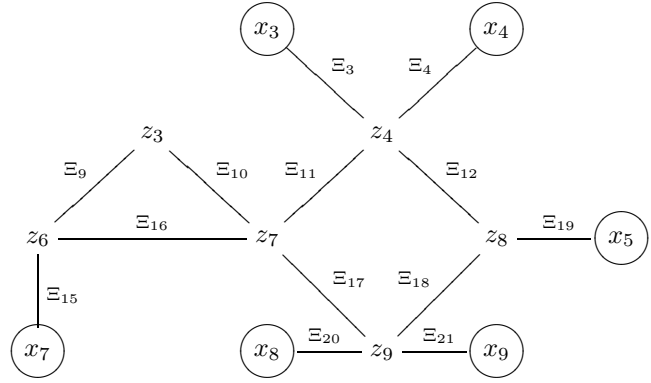

\begin{lemma}\label{UbratTrivialnie}


Suppose $\Phi\in GF(G),$ $MaxUnion(\Phi)=1.$
Then there exists a tree-formula $\Phi'\in GF(G).$

\end{lemma}

\begin{proof}

Let $n(\Phi)$ be the number of unions $U\in Unions(\Phi)$ such that $|U|>1.$
Let us prove this lemma by induction on $n(\Phi).$
If $n(\Phi)=0$ then $\Phi$ is a tree-formula and there is nothing to prove.

Assume that $n(\Phi)>0.$
Hence there exists a union $U\in Unions(\Phi)$ such that $|U|>1.$
Since $MaxUnion(\Phi)=1,$
$U$ is a trivial union.

Let $Z=\{z_{1},\ldots,z_{n}\} = Var(U).$
Since $U$ is trivial, there exist $i$ and $j$ such that
$$KeySet_{\Phi}(U, z_{i}) \cup KeySet_{\Phi}(U, z_{j}) = UVar(\Phi).$$
Without loss of generality it can be assumed that $i = 1,$ $j = 2.$
Let
$$\rho(z_{1},z_{2}) = \exists z_{3}\exists z_{4}\ldots\exists z_{n}\;
U \wedge \bigwedge \limits_{3\le i\le n} Part(\Phi,U,z_{i}).$$
Suppose $\Phi_{1}$ is obtained from $\Phi$ by replacing $U$ by $\rho(z_{1},z_{2}).$
It can be easily checked that $\Phi_{1}$ realizes the same predicate as $\Phi.$
Note that all other unions in $\Phi$ are not changed.
Hence, $n(\Phi_{1})=n(\Phi)-1,$
and $|U'|=1$ for every $U'\in NUnions(\Phi_{1}).$
Hence, by the inductive assumption
there exists a tree-formula $\Phi'\in GF(G).$
This completes the proof.

\end{proof}

\begin{theorem}\label{estTreeFormula}

There exists a tree-formula in $GF(G).$

\end{theorem}

\begin{proof}

It follows from Lemma \ref{VycherkivanieSvobodnih} that
the set $GF(G)$ is not empty.
Suppose $\Phi_{0}\in GF(G).$
Let us construct a consequence
$\Phi_{0},\Phi_{1},\ldots,\Phi_{r}$
such that
$\Upsilon(\Phi_{i+1})< \Upsilon(\Phi_{i})$ for every $i,$
$MaxUnion(\Phi_{r}) = 1.$

Suppose we have $\Phi_{i}.$
If $MaxUnion(\Phi_{i})=1,$  then we put $i=r$ and the consequence is complete.
Suppose $MaxUnion(\Phi_{i})>1.$
Suppose $U\in Max(\Phi_{i})$
and
$MaxKeySet_{\Phi}(U) = MinMaxKeySet(\Phi_{i}).$
Let $\Psi\in U,$ and suppose
a variable $y\in Var(\Psi)$ is used in $U$ at least twice.
Obviously such variable exists.
Put $\Phi_{i+1} =
Trans(\Phi_{i}, y, \Psi).$

Let us prove that $\Upsilon(\Phi_{i+1})<\Upsilon(\Phi_{i}).$
By Lemma \ref{estPredok},
we have
$$MaxUnion(\Phi_{i+1}) \le MaxUnion(\Phi_{i}).$$
If $MaxUnion(\Phi_{i+1}) < MaxUnion(\Phi_{i}),$ then
$\Upsilon(\Phi_{i+1})<\Upsilon(\Phi_{i}).$
If
$$MaxUnion(\Phi_{i+1}) = MaxUnion(\Phi_{i}),\;\;
MaxPar(\Phi_{i+1}) < MaxPar(\Phi_{i}),$$
then $\Upsilon(\Phi_{i+1})<\Upsilon(\Phi_{i}).$

Suppose
$$MaxUnion(\Phi_{i+1}) = MaxUnion(\Phi_{i}),\;\;
MaxPar(\Phi_{i+1}) = MaxPar(\Phi_{i}).$$
By Lemma~\ref{estPredok}, Lemma \ref{UnionNeSoderjitPsi}
and Lemma \ref{UnionSoderjitPsi},
it follows that
for every $U'\in Max(\Phi_{i+1})$ there is $U_{0}\in Max(\Phi_{i})$
such that $U_{0}$ is a parent of $U'.$
Using Lemma~\ref{unionMaxKeySet},
we get $$MaxKeySet_{\Phi}(U')\ge MaxKeySet_{\Phi}(U_{0}).$$
Hence $$MinMaxKeySet(\Phi_{i+1})\ge MinMaxKeySet(\Phi_{i}).$$
If $MinMaxKeySet(\Phi_{i+1})>MinMaxKeySet(\Phi_{i})$
then $\Upsilon(\Phi_{i+1})<\Upsilon(\Phi_{i}).$

Suppose $$MinMaxKeySet(\Phi_{i+1})=MinMaxKeySet(\Phi_{i}).$$
Let $W= \{U_{1}',U_{2}',\ldots,U_{s}'\}$ be the set of all $U'\in Max(\Phi_{i+1})$
such that $$MaxKeySet_{\Phi_{i+1}}(U') =MinMaxKeySet(\Phi_{i+1}).$$
As it was proved earlier,
for every $j$
there is $U_{j}\in Max(\Phi_{i})$
such that $U_{j}$ is a parent of $U_{j}'.$
By Lemma \ref{unionMaxKeySet},
it follows that
$$MaxKeySet_{\Phi_{i}}(U_{j}) \le MaxKeySet_{\Phi_{i+1}}(U_{j}')$$
for every $j.$
Hence, 
\begin{multline*}MinMaxKeySet(\Phi_{i})\le MaxKeySet_{\Phi}(U_{j}) \le \\ \le
MaxKeySet_{\Phi}(U_{j}') = MinMaxKeySet(\Phi_{i+1}).\end{multline*}
So we have $MaxKeySet_{\Phi}(U_{j}) = MinMaxKeySet(\Phi_{i}).$
Also by Lemma \ref{unionMaxKeySet},
we get $U_{j}\neq U_{l}$ for all $j\neq l.$
By Lemma \ref{UnionSoderjitPsi}, it follows that
$U_{j}\neq U$ for every $j.$
Therefore $Number(\Phi_{i+1})<Number(\Phi_{i})$
and $\Upsilon(\Phi_{i+1})<\Upsilon(\Phi_{i}).$

By Lemma \ref{posledovatelnostKonechna},
the sequence $\Phi_{1},\Phi_{2},\Phi_{3},\ldots$ can not be infinite.
Hence there exists $r$ such that $MaxUnion(\Phi_{r})=1$.
By Lemma \ref{UbratTrivialnie}, there exists a tree-formula $\Phi'\in GF(G).$


\end{proof} 

%% file: proofs_main.tex
\section{Proof of the main theorems}

Let us prove two theorems from Section 2.
The next theorem is very similar to the following simple statement from graph theory:
suppose a tree has $p$ vertices and a degree of every vertex is less than $k,$ 
then there exists a simple path in the tree which length is greater than $\log_{k}p.$

\newcounter{backup}
\setcounter{backup}{\value{theorem}}
\setcounter{theorem}{\value{dveteoremy}}

\begin{theorem}

Suppose $SR_{k}\subseteq G,$
$|[G]\cap \widetilde R_{k}|<\infty,$
$ar([G]\cap \widetilde R_{k})=p,$
$ar(G)=q,$
then there exist $\rho \in [G]\cap \widetilde R_{k},$
$\rho_{1},\rho_{2},\ldots,\rho_{n}\in [G]$
such that
\begin{multline*}\rho(x_{1},\ldots,x_{n}) = \exists y_{1}\exists y_{2}\ldots\exists y_{n-1}\;
\rho_{1}(x_{1},y_{1})\wedge
\rho_{2}(y_{1},x_{2},y_{2})\wedge \\ \wedge
\rho_{3}(y_{2},x_{3},y_{3})\wedge \ldots \wedge
\rho_{n-1}(y_{n-2},x_{n-1},y_{n-1})\wedge
\rho_{n}(y_{n-1},x_{n})
\end{multline*}
and
$ar(\rho)>\log_{k\cdot q}(p).$

\end{theorem}

\begin{proof}

If $q=1$ then it can be easily checked that $p = 2$ and
there is nothing to prove.
Suppose $q\ge 2.$ Then without loss of generality it can be assumed that
$[G]\cap (R_{k}^{1}\cup R_{k}^{2}) \subseteq G.$

By Lemma \ref{estTreeFormula}, there exists a tree-formula $\Phi\in GF(G).$
Suppose this formula realizes an essential predicate $\sigma_{1}.$
Suppose $\{t_{1},t_{2},\ldots,t_{p}\} = UVar(\Phi).$
Let $z_{1} = t_{1},$
$\Delta_{1} = \sigma_{k}^{=}(x_{1},z_{1}),$
$$\Xi_{1} = \exists t_{2}\ldots\exists t_{p}\;
\sigma_{k}^{=}(x_{2},t_{2})\wedge \ldots \wedge
\sigma_{k}^{=}(x_{p},t_{p})\wedge \Phi.$$
Put $\Phi_{1} = \exists z_{1}\; \Delta_{1}\wedge \Xi_{1}.$
It can be easily checked that
$\Phi_{1}$ is a tree-formula from $GF(G),$ $\Phi_{1}$ realizes $\sigma_{1}.$

Let us define three sequences
$\Delta_{1},\Delta_{2},\ldots,$
$\Xi_{1},\Xi_{2},\ldots,$
$z_{1},z_{2},\ldots$ inductively.

Suppose we have $\Delta_{i},\Xi_{i},z_{i},$
$\Phi_{i} = \exists z_{i}\; \Delta_{i}\wedge \Xi_{i}.$
Let $\{u_{1},\ldots,u_{s}\} = UVar(\Phi_{i}),$
$\Phi_{i}(u_{1},\ldots,u_{s}) = \sigma_{i}(u_{1},\ldots,u_{s}),$
$\sigma_{i}$ is an essential predicate.
By Lemma \ref{sushnabor}, there exists an essential tuple $(a_{1},\ldots,a_{s})$ for $\sigma_{i}.$
Hence there exists a mapping $\phi:UVar(\Phi_{i})\rightarrow E_{k}$ such that
$\phi(u_{i}) = a_{i}.$

Suppose $z_{i}$ is used in occurrences $\Psi_{1},\ldots,\Psi_{r}$ in $\Xi_{i}.$
Let $Z_{j} = Var(\Psi_{j})\setminus \{z_{i}\}.$
Let $$L_{j} = \bigcup\limits_{z\in Z_{j}} KeySet_{\Phi_{i}}(\{\Psi_{j}\},z).$$
Let $d$ be the number of nonempty sets in $\{L_{1},\ldots,L_{r}\}.$
It follows from Lemma~\ref{nebolshek} that $d\le k.$
Since $|Z_{j}|<q$ for every $j,$
there exists $j_{0}$ and $z\in Z_{j_{0}}$ such that
$|KeySet_{\Phi_{i}}(\Psi_{j_{0}},z)|\ge(|UVar(\Xi_{i})|-1) /(k\cdot q).$

Assume that $KeySet_{\Phi_{i}}(\Psi_{j_{0}},z) = UVar(\Xi_{i})\setminus\{z_{i}\}.$
Put $$\{w_{1},\ldots,w_{l}\} = Z_{1}\cup Z_{2}\cup\ldots \cup Z_{r}\setminus \{z\},$$
$$\rho'(z_{i},z) =
\exists w_{1}\ldots \exists w_{l} \;
\bigwedge\limits_{j\in \{1,2,\ldots,r\}} \Psi_{j} \wedge
\bigwedge\limits_{\substack{j\in \{1,2,\ldots,r\}\\ t\in Z_{j}\setminus \{z\}}} Part(\Phi, \Psi_{j}, t).$$
Put $$\Delta_{i+1} =
\exists z_{i}\;
\Delta_{i}\wedge \rho'(z_{i},z),$$
$$\Xi_{i+1} = Part(\Xi_{i}, \Psi_{j_{0}}, z).$$

Assume that $|KeySet_{\Phi_{i}}(\Psi_{j_{0}},z)| < |UVar(\Xi_{i})|-1.$
Then there exists $$\tau \in Uvar(\Xi_{i})\setminus (\{z_{i}\}\cup KeySet_{\Phi_{i}}(\Psi_{j_{0}},z)).$$
Put $$\{w_{1},\ldots,w_{l}\} = Z_{1}\cup Z_{2}\cup\ldots \cup Z_{r} \cup
UVar(\Xi_{i})\setminus (KeySet_{\Phi_{i}}(\Psi_{j_{0}},z)\cup\{\tau,z_{i}\}).$$
$$\{\zeta_{1},\ldots,\zeta_{q}\} =
UVar(\Xi_{i})\setminus (KeySet_{\Phi_{i}}(\Psi_{j_{0}},z)\cup\{\tau\}).$$
Let
\begin{multline*}\rho'(z_{i},\tau,z) =
\exists w_{1}\ldots \exists w_{l} \;
(\zeta_{1} = \phi(\zeta_{1})) \wedge
\ldots
\wedge
(\zeta_{q} = \phi(\zeta_{q})) \wedge \\ \wedge
\bigwedge\limits_{j \in \{1,2,\ldots,r\}} \Psi_{j} \wedge
\bigwedge\limits_{\substack{j\in \{1,2,\ldots,r\}\\ t\in Z_{j}\setminus \{z\}}} Part(\Phi, \Psi_{j}, t).\end{multline*}

Put $$\Delta_{i+1} =
\exists z_{i}
\Delta_{i} \wedge \rho'(z_{i},\tau,z),$$
$$\Xi_{i+1} = Part(\Xi_{i}, \Psi_{j_{0}}, z).$$

We suppose that
$\Xi_{i+1}$ is not defined if $\Xi_{i}$ contains just one predicate.
It can be easily checked that
this predicate is $\sigma_{k}^{=}.$
In this cases we say that $\Xi_{i}$ is the last member of the sequence.

It can be easily checked that
$\Xi_{i+1}$ is shorter then $\Xi_{i}.$
Hence, this sequence can not be infinite.
Suppose $\Delta_{f},$ $\Xi_{f},$ $z_{f}$ are the last members of these sequences.

It follows from the definition that
for every $i$
one of the following conditions hold

\begin{enumerate}

\item 
$|UVar(\Delta_{i+1})| = |UVar(\Delta_{i})|,$
$|UVar(\Xi_{i+1})| = |UVar(\Xi_{i})|.$

\item $|UVar(\Delta_{i+1})| = |UVar(\Delta_{i+1})|+1,$
$|UVar(\Xi_{i+1})|-1 \ge (|UVar(\Xi_{i})|-1)/(k\cdot q).$
\end{enumerate}

Since $|UVar(\Xi_{1})| = p,$
$|UVar(\Delta_{1})| = 2,$
it follows that $$|UVar(\Delta_{f})|\ge log_{k\cdot q}(p-1).$$

Suppose there exists a subformula in $\Phi_{f}$
of the form
$\exists t_{3}\;\delta(t_{1},t_{2},t_{3})\wedge \gamma(t_{3},t_{4}),$
then it can be replaced by
$\delta'(t_{1},t_{2},t_{4}) = \exists t_{3}\;\delta(t_{1},t_{2},t_{3})\wedge \gamma(t_{3},t_{4}).$

Using this transformation and by renaming variables in $\Phi_{f}$
we get a formula of the following form
\begin{multline*}\rho(x_{1},\ldots,x_{n}) = \exists y_{1}\exists y_{2}\ldots\exists y_{n-1}\;
\rho_{1}(x_{1},y_{1})\wedge
\rho_{2}(y_{1},x_{2},y_{2})\wedge \\ \wedge
\rho_{3}(y_{2},x_{3},y_{3})\wedge \ldots \wedge
\rho_{n-1}(y_{n-2},x_{n-1},y_{n-1})\wedge
\rho_{n}(y_{n-1},x_{n}).
\end{multline*}
where $n = |UVar(\Delta_{f})| +1 \ge log_{k\cdot q}(p-1) +1 > log_{k\cdot q}p.$
Moreover,
it can be easily checked that
$(\phi(x_{1}),\phi(x_{2}),\ldots,\phi(x_{n}))$ is an essential tuple for $\rho.$
Hence, $\rho$ is an essential predicate.
This completes the proof.

\end{proof}

\begin{theorem} 

Suppose $\rho \in \widetilde R_{k},$ $G\subseteq R_{k},$ $\rho_{1},\rho_{2},\ldots,\rho_{n}\in [G]$
\begin{multline*}\rho(x_{1},\ldots,x_{n}) = \exists y_{1}\exists y_{2}\ldots\exists y_{n-1}\;
\rho_{1}(x_{1},y_{1})\wedge
\rho_{2}(y_{1},x_{2},y_{2})\wedge \\ \wedge
\rho_{3}(y_{3},x_{3},y_{3})\wedge \ldots \wedge
\rho_{n-1}(y_{n-2},x_{n-1},y_{n-1})\wedge
\rho_{n}(y_{n-1},x_{n})
\end{multline*}
where $n>2^{2k^{2}}+2.$
Then $|[G] \cap \widetilde R_{k}| = \infty.$
\end{theorem}

\begin{proof}

Assume that $|[G] \cap \widetilde R_{k}| < \infty.$
Hence, it can be assumed that
$\rho$ has the maximal arity among all essential predicates that can be presented by the formula
from the condition of the theorem.
Suppose $(a_{1},\ldots,a_{n})$ is an essential tuple for $\rho.$
Therefore, there exist $b_{1},b_{2},\ldots,b_{n}\in E_{k}$ such that
$$\rho(a_{1},\ldots,a_{i-1},b_{i},a_{i+1},\ldots,a_{n})=1$$ for every $i.$
Let
\begin{multline*}\sigma_{m}(y_{1},x_{2},\ldots,x_{m-1},y_{m}) = \\ =
\exists y_{2}\exists y_{3}\ldots\exists y_{m-1}\;
\rho_{2}(y_{1},x_{2},y_{2})\wedge
\rho_{3}(y_{2},x_{3},y_{3})\wedge \ldots \wedge
\rho_{m-1}(y_{m-1},x_{m},y_{m}).
\end{multline*}
Put
$$D_{m}= \{(d_{1},d_{m}) \; | \: \sigma_{m}(d_{1},a_{2},\ldots,a_{m},d_{m})=1\},$$
$$F_{m}= \{(d_{1},d_{m}) \; | \: \exists i\; \sigma_{m}(d_{1},a_{2},\ldots,a_{i-1},b_{i},a_{i+1},\ldots,a_{m},d_{m})=1\}.$$

Since $n>2^{2k^{2}}+2,$
there exists $p,q\in \{2,\ldots,n\},1<p<q<n,$
such that $D_{p}=D_{q}$ and $F_{p}=F_{q}.$
Let
\begin{multline*}
\rho'(z_{1},z_{2}\ldots,z_{q},x_{p+1},\ldots,x_{n}) =
\exists t_{1}\exists y_{p}\exists y_{p+1}\ldots\exists y_{n-1}\;
\rho_{1}(z_{1},t_{1})\wedge \\ \wedge
\sigma_{q}(t_{1},z_{1},\ldots,z_{q},y_{p}) \wedge
\rho_{p+1}(y_{p},x_{p+1},y_{p+1})\wedge \ldots \\ \ldots \wedge
\rho_{n-1}(y_{n-2},x_{n-1},y_{n-1})\wedge
\rho_{n}(y_{n-1},x_{n}) = \\
= \exists t_{1}\exists t_{2}\ldots \exists t_{q-1}\exists y_{p}\exists y_{p+1}\ldots\exists y_{n-1}\;
\rho_{1}(z_{1},t_{1})\wedge \\ \wedge
\rho_{2}(t_{1},z_{2},t_{2})\wedge
\rho_{3}(t_{2},z_{3},t_{3})\wedge \ldots \wedge
\rho_{q-1}(t_{q-2},z_{q-1},t_{q-1})\wedge
\rho_{q}(t_{q-1},z_{q},y_{p})\wedge \\ \wedge
\rho_{p+1}(y_{p},x_{p+1},y_{p+1})\wedge \ldots \wedge
\rho_{n-1}(y_{n-2},x_{n-1},y_{n})\wedge
\rho_{n}(y_{n-1},x_{n})
\end{multline*}

It can be easily checked that
$(a_{1},a_{2},\ldots,a_{q-1},a_{q}, a_{p+1},a_{p+2},\ldots,a_{n})$ is an essential tuple
for $\rho'.$ Therefore $\rho'$ is an essential predicate. This contradicts the assumption about the maximality of
the arity of $\rho.$ The theorem is proved.

\end{proof}

\setcounter{theorem}{\value{backup}}